\documentclass{article}
\usepackage[utf8]{inputenc}
\usepackage[T1]{fontenc}

\usepackage{lmodern}
\usepackage{amsmath}
\usepackage{amsfonts}
\usepackage{amssymb}
\usepackage{amsmath}
\usepackage{amsthm}
\usepackage[a4paper,left=3.3cm,right=3.3cm,bottom=2.2cm,top=2cm,foot=1.3cm,includeheadfoot=true]{geometry}
\usepackage{framed}
\usepackage[boxed]{algorithm2e}
\usepackage[hidelinks]{hyperref}
\usepackage{multirow,bigdelim}
\usepackage[babel=true]{csquotes}
\DeclareMathOperator{\Z}{\mathbb{Z}}
\DeclareMathOperator{\C}{\mathcal{C}}
\DeclareMathOperator{\MDC}{\mathrm{degmax}}
\DeclareMathOperator{\ord}{\mathrm{ord}}
\newcommand{\softO}{\widetilde{O}}
\newtheorem{lemma}{Lemma}
\newtheorem{proposition}[lemma]{Proposition}
\newtheorem{definition}[lemma]{Definition}
\newtheorem{theorem}[lemma]{Theorem}

\newcommand{\N}{\mathbb Z_{\geq 0}}

\newcommand{\F}{\mathbb F}

\newcommand{\Og}{O_g}

\DeclareMathOperator{\Res}{Res}

\DeclareMathOperator{\Err}{Err}
\DeclareMathOperator{\aff}{aff}

\newcounter{countproblem}
\newenvironment{problem}[1] 
{ 
  \refstepcounter{countproblem} 
  \begin{center} 
    \begin{minipage}{0.9\linewidth} 
      \begin{framed} 
        {\bf {#1}.} 
} 
{   
      \end{framed} 
    \end{minipage} 
  \end{center} 
}

\begin{document}
\title{Improved Complexity Bounds for Counting Points on Hyperelliptic Curves}
\author{Simon Abelard, Pierrick Gaudry and Pierre-Jean Spaenlehauer \\[0.2cm]
{\normalsize Université de Lorraine, CNRS, Inria}}
\date{}
\maketitle

\begin{abstract}
  We present a probabilistic Las Vegas algorithm for computing the local zeta
  function of a hyperelliptic curve of genus $g$ defined over $\mathbb
  F_q$.
  It is based on the approaches by Schoof and Pila combined
  with a modelling of the $\ell$-torsion by structured polynomial
  systems.
  Our main result improves on previously known complexity bounds by
  showing
  that there exists a constant $c>0$ such that, for any fixed $g$, this
  algorithm has expected time and space complexity $O((\log q)^{c g})$ as
  $q$ grows and the characteristic is large enough.
\end{abstract}

\section{Introduction} 
Since the discovery of Schoof's algorithm \cite{Sc85}, the complexity of counting 
points on curves and Abelian varieties defined over finite fields has attracted a lot of attention due to its numerous applications in
cryptology, number theory and algebraic geometry.
In this paper, we investigate the complexity of computing the local zeta
function of hyperelliptic curves of fixed large genus. We propose a probabilistic
algorithm which relies on the same foundations as Schoof's \cite{Sc85} and
Pila's algorithms \cite{Pi90}.

When the characteristic of the base field $\mathbb F_q$ is small, Kedlaya's and Satoh's
approaches
\cite{Ked01,Satoh00} and their variants compute very efficiently the number of
rational points of Jacobians of hyperelliptic curves. We can also mention
Lauder-Wan's~\cite{LaWa02} and Lauder's~\cite{Lauder04} methods that can handle very general
varieties. The current best algorithms in this family for rather general curves are by
Tuitman~\cite{Tuitman16,Tuitman17}.
However, the complexities
of these $p$-adic algorithms are exponential in $\log(p)$, where $p$ is the
characteristic of the base field. This dependency can be made as low as
$\sqrt{p}$, thanks to the work of Harvey~\cite{Harvey15}.
Another line of research aims at taking profit of
extra structure of the curve, assuming that this structure is known in advance and described in a
convenient way. 
The most popular case is the Complex Multiplication method~\cite{AtMo93},
and in~\cite{gaudry2011counting} it is shown how to exploit
real multiplication for counting points on genus 2 curves. When there is no
such explicitly known additional structure and the characteristic of the base field is
large, Schoof-Pila's $\ell$-adic algorithms are the main tools for counting
points.

These $\ell$-adic methods were introduced for elliptic curves in \cite{Sc85}, and
later extended to Abelian varieties in \cite{Pi90}. In particular, Pila showed
that the local zeta function of a $g$-dimensional Abelian variety can be computed within $O(\log(q)^\Delta)$
operations, where $\Delta$ and the constant in the $O()$ are functions of $g$
(but they do not depend on $q$). This complexity result requires some
assumptions on the presentation of the Abelian variety which are
satisfied by Jacobians of hyperelliptic curves given via a
Weierstrass form~\cite{Pila05}.  Complexity
improvements were obtained in \cite{HI98} and \cite{AH01}. The latter article gives
a deterministic algorithm for counting points on hyperelliptic curves with
complexity $(\log q)^{O(g^2\log g)}$. Pila's algorithm and its variants may
differ from Schoof's algorithm when specialized to the case of elliptic curves, but they are
nonetheless related because they all rely on computing the characteristic
polynomial $\chi$ of the Frobenius endomorphism $\varphi$ modulo a prime number
$\ell$ for sufficiently many such primes to deduce the numerator of the local zeta function
of the curve (which is in fact the reciprocal polynomial of $\chi$).

More precisely, let $\mathcal C$ be a hyperelliptic curve of genus $g$ and $J$
be its Jacobian. When $\ell$ is a prime different from the characteristic
of the base field, the $\ell$-torsion group $J[\ell]$ is isomorphic to $(\Z/ \ell\Z)^{2g}$ and the characteristic polynomial of
the restriction of $\varphi$ on $J[\ell]$ is exactly $\chi \bmod \ell$.
Furthermore, $\chi(1)=\# J(\F_q)$. The principle of Schoof-Pila's algorithm is
to pick elements $D$ in $J[\ell]$ and to find conditions on the coefficients
$s_0, \ldots, s_{2g-1}\bmod \ell$ such that
$\varphi^{2g}(D)+s_{2g-1}\varphi^{2g-1}(D)+\cdots+s_0D$ is equal to $0$
in $J[\ell]$. By testing all the tuples $(s_0,\ldots, s_{2g-1})$ up to
the symmetries coming from the functional equation of $\chi$
(and possibly many $D$), the number of
possibilities for $\chi\bmod\ell$ is reduced until only one remains. The
numerator of the zeta function is then obtained by repeating
this procedure for many $\ell$ and by using
Weil's conjectures to bound the absolute value of the coefficients.

For such a strategy, it is of the utmost importance that we get a description
of the $\ell$-torsion for which computations are reasonably easy to perform. In
the elliptic case, computations in the $\ell$-torsion subgroup are achieved by
computing in the ring $\F_q[X]/ \psi_{\ell}(X)$ where $\psi_{\ell}$ is
the $\ell$-division polynomial, which has  degree
$O(\ell^2)$. The dominant part of the complexity is the computation of
$\varphi^2(X)$ in this quotient ring. In the genus 2 case, the bottleneck of the
algorithm is no longer the computation of the powers of $\varphi$ but that of
a convenient algebraic representation of the $\ell$-torsion \cite{GS12}.
This appears to be also the case for $g>2$. In order to reach the
desired complexity, our main task is to compute such a representation
efficiently. This is the central part of the proof of the complexity bound, and
it is obtained by combining a special modelling of the $\ell$-torsion with the 
geometric resolution algorithm \cite{GiuLecSal01}, and by using 
multi-homogeneous B\'ezout bounds. More precisely, we show how to
construct a polynomial system whose solutions are the $\ell$-torsion
points. This system involves two sets of variables, the first containing
a small number $O(g)$ of variables, each of them occurring with a degree
that is polynomial in $\ell$, and the second set containing many more variables
but all of them occur with a degree that can be bounded independently of
$\ell$. This bi-homogeneous structure is the key to obtain a complexity bound that is
better than for an unstructured system with the same number of variables
and the same degree.

Another important ingredient in the proof of our main result is the
extension of
degree bounds for the coefficients of Cantor's analogue to division
polynomials~\cite{Ca94}. Indeed, these polynomials are involved in the
modelling of the $\ell$-torsion and the degrees of their coefficients have a
direct impact on the complexity of solving the polynomial system representing
the $\ell$-torsion.

We finally mention that our result is of a purely theoretical nature. In the
case of genus 2 and 3, the geometric resolution algorithm is at
best quadratic in the degree of the $\ell$-torsion ideal, which brings no
improvement over a more direct study of the polynomial systems describing the
$\ell$-torsion. And for curves of larger (fixed) genus, we are still far from a
situation where practical experiments could be run.

\paragraph{Organization of the paper.} Section \ref{sec:schoof} describes a
general algorithm for point-counting on Abelian varieties along with its
complexity, assuming that the $\ell$-torsion can be efficiently computed.
Section \ref{sec:polsys} establishes the complexity result for
multi-homogeneous polynomial systems that is required to obtain our claimed
complexity bound. Section \ref{sec:generic} contains the modelling of the
$\ell$-torsion under some mild assumptions on its structure. 
Finally, Section~\ref{sec:nongeneric} describes the complete
modelling of the $\ell$-torsion, which is faithful even if the assumptions
required in Section \ref{sec:generic} are not satisfied.

\paragraph{Acknowledgements.} We are grateful to \'Eric Schost and Guillermo
Matera for fruitful discussions and for pointing out important references. We
also wish to thank anonymous referees for their comments which helped improve the
paper.

\section{Overview of the main result}\label{sec:schoof}

Our main result is a probabilistic algorithm and a complexity bound for solving the following problem.

\begin{problem}{Computing local zeta functions of hyperelliptic curves}
  Given an odd prime power $q$, a positive integer $g$ and a squarefree univariate polynomial
  $f\in\mathbb F_q[X]$ of degree $2\, g+1$, let $\C$ be the
  hyperelliptic curve with Weierstrass form $Y^2 = f(X)$. Compute the numerator
  $P_{\C}\in\mathbb
  Z[T]$ of the local zeta function of $\C$:
  $$Z(\C/\mathbb F_q, T) = \exp\left(\sum_{i=1}^\infty\#\C(\mathbb
  F_{q^i})\cdot \dfrac{T^i}{i}\right) = \dfrac{P_{\C}(T)}{(1-T)(1-qT)}.$$
\end{problem}

The special form of the denominator of the local zeta function is a consequence
of Weil's conjectures. We refer to \cite[Ch.~XI,
Thm.~5.2]{lorenzini1996invitation} for more details. 
Throughout the paper, we
shall assume that the
characteristic of $\mathbb F_q$ is sufficiently large compared to $\log q$. This
assumption is required by a variant of Bertini's theorem
(Proposition~\ref{prop:regred}). 

Our main result is as follows.

\begin{theorem}\label{th:main} 
  There exists an explicitly computable constant
  $c$ such that for all genus $g$, there exists an integer $q_0(g)$ such that
  for all prime power $q=p^n$ larger than $q_0(g)$ with $p\geq (\log
  q)^{cg}$ and for all hyperelliptic curves
  $\C$ of genus $g$ defined over $\F_q$, the numerator $P_{\C}$ of the local zeta
  function of $\C$ can be computed with a probabilistic algorithm in expected
  time bounded by $(\log q)^{cg}$.
\end{theorem}

This complexity result is summarized by the notation $O_g((\log
q)^{O(g)})$,
keeping in mind that $g$ is fixed and $q$ grows to infinity. Indeed, 
such a complexity statement can hide any factor that
depends only on $g$: a
running time in $f(g)(\log q)^{cg}$ can be transformed into $(\log
q)^{c'g}$ by taking a value $c'$ larger than $c$ and adjusting $q_0(g)$,
so that $|f(g)| \le (\log q_0(g))^{(c'-c)g}$.

A typical example used in this article is the multiplication of two
polynomials of degree $d=(\log q)^{O(g)}$. Using FFT-based techniques,
this can be done in $\softO(d)$ operations, which can be rewritten as
$(\log q)^{O(g)}(\log((\log q)^{O(g)}))^k$ for some constant $k$ and is
therefore again in $O_g(\log(q)^{O(g)})$. Here the function $f(g)$ that has
been hidden in the operation is polynomial in $g$, but we will have cases
where it is a combinatorial factor that grows very quickly with $g$ and
we make no effort to optimize it.

A classical geometrical object associated to a genus $g$ curve is its Jacobian
variety. Over the algebraic closure of $\mathbb F_q$, it can be described as
the multiset of at most $g$ points of the curve and it is endowed with an
Abelian group structure (it is isomorphic to the degree-$0$ subgroup of the
Picard group of the curve). The Frobenius map acts in a natural way on this
Jacobian and it is compatible with its $\Z$-module structure.

Throughout this paper, $\C$ is a hyperelliptic curve defined over $\mathbb F_q$ with at least one rational
Weierstrass point. Hence $\C$ admits a Weierstrass model $y^2=f(x)$, where $f$ is a
squarefree monic polynomial of degree $2g+1$. 
If $\C$ does not have any rational Weierstrass point, then
we can extend the base field so that there exists a rational Weierstrass
point that we send to infinity. The degree of the extension does not
depend on $q$ (it is at most linear in $g$), so that this will not affect
our complexity result.

For practical computations, we need a coordinate system to represent points on
the Jacobian of $\C$: they
shall be encoded via their Mumford
representation using $2g$ coordinates. The group law on points in the Jacobian can be performed
with Cantor's algorithm~\cite{cantor1987computing} which operates with elements in Mumford
representation at a cost of $\softO(g)$ base field operations.

The algorithm that allows to prove the theorem is essentially the same as
the one proposed by Pila for Abelian varieties, which is itself inspired
by Schoof's algorithm for counting points on elliptic curves.
This algorithm relies on a few classical results for curves defined over finite
fields:

\begin{itemize}
  \item The numerator $P_{\C}$ of the local zeta function is the reciprocal of the characteristic
    polynomial of the Frobenius morphism on the Jacobian variety $J$ of
    $\C$ \cite[Thm.~5.2]{lorenzini1996invitation};
  \item For prime numbers $\ell$ not dividing $q$, the $\ell$-torsion $J[\ell]$ of the
    Jacobian variety is isomorphic (as an Abelian group) to
    $(\Z/\ell\Z)^{2g}$ \cite[Sec.~II.6, Prop.~page 64]{mumford1974abelian}, \cite[Thm.~4.73]{cohen2005handbook}; Therefore $P_{\C}\bmod \ell$ is the
    reciprocal of the characteristic polynomial of the Frobenius seen as an
    endomorphism of $J[\ell]\cong (\Z/\ell\Z)^{2g}$;
  \item The Weil conjectures
    imply that $P_{\C}$ has the following form over the complex numbers: $P_{\C}(T) = \prod_{i=1}^{2g} (1 - u_i T)$
    with
    $\lvert u_i\rvert = q^{1/2}$ \cite[Ch.~VIII, Thm.~6.1]{lorenzini1996invitation}. Moreover, if $a_0,\ldots, a_{2g}$ denote the
    coefficients of $P_{\C}$, the functional equation implies that
    $a_{2g-i} = q^{g-i} a_i$. Consequently, the absolute value of the
        coefficients of $P_{\C}\in\Z[T]$ are bounded
    by $\binom{2g}{g}q^g$.
\end{itemize}

\begin{algorithm}[ht] \label{algo:pila}
  \KwData{$q\in\mathbb Z_{>0}$ a prime power, and $f\in\mathbb F_q[X]$ a monic squarefree
    univariate polynomial.}
    \KwResult{The characteristic polynomial $\chi\in\mathbb Z[T]$ of the Frobenius
    endomorphism on the Jacobian $J$ of the hyperelliptic curve defined over
    $\mathbb F_q$ with
  Weierstrass form $Y^2 = f(X)$.}
 
  $\ell \gets 1$\;
  $R\gets 1$\;

  \While{$R\leq 2\binom{2g}{g}q^g + 1$} {
	$\ell\gets${\sf NextPrime}$(\ell)$\;
	\If{$\ell$ divides $q$}{$\ell\gets${\sf NextPrime}$(\ell)$\;}
        Compute a description of $J[ \ell]$\;
	Compute a $2g\times 2g$ matrix $F$ with coefficients in $\Z/\ell\Z$
	representing the action of the Frobenius on $J[\ell]\cong
	(\Z/\ell\Z)^{2g}$\;
        Compute the characteristic polynomial $\chi \bmod \ell$ of the matrix
	$F$\;
	$R \gets R\cdot\ell$\;
  }
    Reconstruct $\chi$ using the Chinese Remainder Theorem.

    \caption{A bird's eye view of Pila's point counting algorithm for
    hyperelliptic curves.}
\end{algorithm}

Pila's algorithm reconstructs the numerator of the local zeta function of $\C$
by computing the action of the Frobenius on the $\ell$-torsion for
sufficiently-many prime numbers $\ell$ and by using the Chinese Remainder
Theorem. A bird's eye view of this algorithm is
given in Algorithm~\ref{algo:pila}.
The main difficulty resides in the step where one computes an explicit
description of $J[\ell]$. Since $J[\ell]$ is a $0$-dimensional variety of
degree $\ell^{2g}$, what we will compute is a geometric resolution of the
corresponding radical ideal, that is a univariate squarefree polynomial
$F_\ell(T)$, together with $2g$ coordinate polynomials $\gamma_i(T)$,
such that the coordinates of the $\ell$-torsion elements are the evaluations of
the vector $(\gamma_1(T),\ldots, \gamma_{2g}(T))$ at the roots of
$F_\ell$.

To be more precise, the Mumford coordinates are in fact a set of $g$
affine systems of coordinates, each corresponding to a different weight of the
represented divisor (the definition is recalled in
Section~\ref{sec:generic}). The variety $J[\ell]$ will accordingly be
represented by a set of $g$ geometric resolutions, each encoding
$\ell$-torsion divisors of a given weight $w\in[1,g]$. Generically, we expect that
all the elements in $J[\ell]$ have weight $g$, except for the neutral element which
has weight 0.
Most of the article is dedicated to computing efficiently this representation
for $J[\ell]$. The cornerstone of the proof of Theorem~\ref{th:main} relies on the following statement.

\begin{proposition}\label{prop:Jl}
    Let $\C$ be a hyperelliptic curve of genus $g$ over $\F_q$ with Weierstrass
    form $Y^2 = f(X)$ ($f$ monic of degree $2g+1$) and $J$
    be its Jacobian variety. Let $\ell>g$ be a prime not dividing $q$. Assuming
    that the characteristic of $\mathbb F_q$ is sufficiently large 
    as in Theorem~\ref{th:main}, there is a
    Las Vegas probabilistic algorithm which takes as input $q, \ell, f$ and which computes geometric
    resolutions for the varieties $\{J_w[\ell]\}_{w\in[1,g]}$ of $\ell$-torsion points of weight
    $w$ in the Jacobian variety. This algorithm can be implemented by a Turing
    machine with space and expected time $O_g\left((\ell\,\log q)^{O(g)}\right)$.
\end{proposition}

Assuming this complexity bound, performing a complexity analysis as done
in~\cite{Pi90} leads to a complexity bound for Algorithm~\ref{algo:pila} that
corresponds to Theorem~\ref{th:main}. We recall it here for completeness,
with some simplifications due to the fact that we consider a probabilistic
algorithm, so we can factor polynomials using Cantor-Zassenhaus' algorithm.

\begin{proof}[Proof of Theorem~\ref{th:main} assuming
    Proposition~\ref{prop:Jl}.]

    By Weil's bounds, the absolute values of the coefficients of the
    characteristic polynomial $\chi$ are bounded by $\binom{2g}{g}q^g$.
    Therefore at the end of the loop of Algorithm~\ref{algo:pila}, these
    coefficients are completely determined by their values modulo all the
    primes $\ell$ that have been explored.  It follows from
    \cite[Cor.~10.1]{tenenbaum1995}
    that the largest $\ell$ in the loop is at most linear in $g\log q$.
    From this and Proposition~\ref{prop:Jl}, computing the description
    of $J[\ell]$ as a union of geometric resolutions for all the
    $J_w[\ell]$ can be achieved within expected complexity $O_g\left((\log
    q)^{O(g)}\right)$.

    Factoring the univariate polynomials involved in the geometric
    resolutions can be done within the same time bound $O_g\left((\log
    q)^{O(g)}\right)$,
    since the sum of their degrees is $\ell^{2g}$ and factoring polynomials in finite
    fields can be done in time linear in $\log(q)$ and quasi-quadratic in the
    degree \cite[Thm.~14.14]{GatGer}. Therefore, it is possible to construct a Mumford
    representation for each $\ell$-torsion divisor within the same complexity, each of them possibly
    defined over a different extension of $\F_q$. In fact, due to the
    rationality of the group law that acts on $J[\ell]$, one of these
    extensions of $\F_q$ contains all the others.

    Using elementary linear algebra for the Frobenius endomorphism
    $\varphi$ acting on $J[\ell]$ (seen as an $\F_\ell$-vector space),
    we can deduce $\chi_\ell = \chi \bmod \ell$.
    We first compute a basis of $J[\ell]$ by brute force and a
    dictionary of how all elements decompose on it. Then, the action of
    $\varphi$ on the basis elements can be computed and the result is a
    matrix whose characteristic polynomial is $\chi_\ell$. All of this
    fits in the $O_g((\log q)^{O(g)})$ complexity bound.
    The loop is repeated $O_g(\log q)$ times, and this additional factor
    does not affect the overall complexity.
\end{proof}

\section{Polynomial systems}\label{sec:polsys}

This section is devoted to describing tools that we will use to estimate the
complexity of computing a convenient representation of the $\ell$-torsion of
the Jacobian of hyperelliptic curves.

We start by fixing some notation.
In the sequel, $\overline{\mathbb F_q}$ denotes the
algebraic closure of $\mathbb F_q$.
For an ideal $I\subset \mathbb F_q[X_1,\ldots, X_n]$, we call
dimension of $I$ and note $\dim(I)$ the Krull dimension of the quotient
ring $\mathbb F_q[X_1,\ldots, X_n]/I$.  Moreover, by identifying a point
$(\lambda_0,\ldots, \lambda_n)\in\overline{\mathbb F_q}^{n+1}$ with the
polynomial $\lambda_0 + \lambda_1 X_1 + \dots +\lambda_n
X_n\in\overline{\mathbb F_q}[X_1,\ldots, X_n]$,
there is a dense Zariski open
subset $\mathcal O\subset(\overline{\mathbb F_q}^{n+1})^{\dim(I)}$ such that for any
$(\ell_1,\ldots, \ell_{\dim(I)})\in\mathcal O$, the algebra
$\overline{\mathbb F_q}[X_1,\ldots, X_n]/(I+\langle \ell_1,\ldots,
\ell_{\dim(I)}\rangle)$ is a finite dimensional $\overline{\mathbb F_q}$-vector space of
constant dimension, which is called the degree of $I$.
A sequence $(f_1,\ldots, f_i)\in
\mathbb F_q[X_1,\ldots, X_n]^i$ is regular if $\langle f_1,\ldots,
f_i\rangle\neq \mathbb F_q[X_1,\ldots, X_n]$ and for any
$j\in[2, i]$, $f_j$ does not divide zero in $\mathbb F_q[X_1,\ldots, X_n]/\langle f_1,\ldots,
f_{j-1}\rangle$. 
The sequence
$(f_1,\ldots, f_i)$ is reduced if every intermediate ideal $\langle
f_1,\ldots, f_j \rangle$ with $j\in[1,i]$ is radical.

\paragraph{Geometric resolutions.} For describing $0$-dimensional (i.e.
finite) sets $V\subset \overline{\mathbb F_q}^n$ where $V$ is defined
over $\F_q$, we use
a data structure called a geometric resolution of $V$. The terminology
here is borrowed from \cite{cafure2006fast}, see also \cite{GiuLecSal01}. An 
$\mathbb F_{q^e}$-geometric resolution of
$V$ is a tuple $( (\ell_1,\ldots, \ell_n), Q, (Q_1,\ldots, Q_n))$ where:
\begin{itemize}
  \item The vector $(\ell_1,\ldots,\ell_n)\in\mathbb F_{q^e}^n$ is such
    that the linear form 
    $$\begin{array}{rccc}\ell : &\overline{\mathbb
      F_q}^n&\rightarrow&\overline{\mathbb F_q}\\
      &(x_1,\ldots, x_n)&\mapsto &\sum_{i=1}^n
      \ell_i x_i\end{array}$$ takes distinct values at all points in $V$. The linear form
    $\ell$ is called the primitive element of the geometric resolution;
  \item The polynomial $Q\in\mathbb F_{q^e}[T]$ equals
    $\prod_{\mathbf x\in V}(T-\ell(\mathbf x));$
  \item The polynomials $Q_1,\ldots, Q_n\in\mathbb F_{q^e}[T]$ parametrize $V$
    by the roots of the polynomial $Q$, i.e.
    $$V = \{(Q_1(t),\ldots, Q_n(t))\mid t\in\overline{\mathbb F_q},  Q(t) = 0\}.$$
\end{itemize}

We note that our definition is slightly simpler than the one in
\cite[Sec.~2.1]{cafure2006fast} because we restrict ourselves to the
$0$-dimensional case in this paper (in \cite[Sec.~2.1]{cafure2006fast}, the
definition is also valid for equidimensional varieties with positive
dimension).

In the following statement, if $f$ is a polynomial in a ring $\mathbb
F_q[X_1,\ldots, X_{n_x}, Y_1,\ldots, Y_{n_y}]$, then we let $\deg_x(f)$ (resp.
$\deg_y(f)$) denote the degree of $f(X_1,\ldots, X_{n_x}, y_1,\ldots,
y_{n_y})\in\overline{\mathbb F_q}[X_1,\ldots,$ $X_{n_x}]$ (resp. $f(x_1,\ldots,
x_{n_x}, Y_1,\ldots, Y_{n_y})\in \overline{\mathbb F_q}[Y_1,\ldots, Y_{n_y}]$),
where $y_1,\ldots, y_{n_y}$ (resp.  $x_1,\ldots, x_{n_x}$) are generic values
in $\overline{\mathbb F_q}$.

The following proposition is a cornerstone of our complexity result for
computing the $\ell$-torsion of the Jacobian of a hyperelliptic curve. The
statement and its proof combine three main ingredients: (1) the geometric
resolution algorithm~\cite{GiuLecSal01} and its version for finite
fields~\cite{cafure2006fast}, which are methods for solving polynomial systems whose
complexity depends mainly on geometric degrees; (2) the
multi-homogeneous Bézout bound which allows us to control the geometric degrees
by separating the variables in our modelling in two blocks, where the block
supporting most of the degrees has small cardinality; (3) a variant of
Bertini's theorem to process our polynomial system into a reduced regular
sequence which is a valid input for the geometric resolution algorithm.

As we shall see in the next sections, our polynomial system modelling the
$\ell$-torsion will have two blocks of variables. The first block occurs
with large degree $\ell^{O(1)}$ but it has very small cardinality in $O(g)$.
The second block has a larger cardinality, but the degrees of the
equations with respect to this block do not depend on $\ell$, but only on
$g$. Taking this bi-homogeneous structure into account is crucial to reach our
claimed complexity bound. The following proposition provides a bound on the complexity of
solving polynomial systems having this structure, and the sequel of this
section is dedicated to its proof.

\begin{proposition}\label{prop:mainpolsys}
  There exists a probabilistic Turing machine $\mathbf T$ which takes as input
  polynomial systems with coefficients in a finite field $\mathbb F_q$
  and which satisfies the following
  property. For any function $h:\Z_{>0}\rightarrow \Z_{>0}$, for any
  positive number $C>0$ and for any $\varepsilon>0$, there exists a function
  $\nu:\Z_{>0}\rightarrow\Z_{>0}$ and a positive number
  $D>0$ such that for all positive integers $g, \ell, n_x, n_y, d_x, d_y, m >
  0$ such that $n_x < C\, g$, $n_y<h(g)$, $d_x<h(g)\, \ell^C$, $d_y<h(g)$,
  $m<h(g)$, for any prime power $q$ such that the prime number $p$ dividing $q$
  satisfies $2^{n_x+n_y}d_x^{n_x}\, d_y^{n_y} < p$, and for any polynomial system $f_1,\ldots, f_m\in
  \mathbb F_q[X_1,\ldots, X_{n_x}, Y_1,\ldots, Y_{n_y}]$ such that 
  \begin{itemize}
    \item for all $i\in
      [1,m]$, $\deg_x(f_i)\leq d_x$ and $\deg_y(f_i)\leq d_y$,
    \item the ideal $I= \langle f_1,\ldots, f_m\rangle$ has dimension $0$ and is
  radical,
  \end{itemize}
  the Turing machine $\mathbf T$ with input $f_1,\ldots, f_m$ returns an
    $\mathbb F_{q^{\lceil\nu(g)\log\ell\rceil}}$-geometric
      resolution of the variety $\{\mathbf x\in\overline{\mathbb F_q}\mid
      f_1(\mathbf
    x) = \dots = f_m(\mathbf x) = 0\}$ with probability at least $5/6$, using
      space and time bounded above by $\nu(g)\,
      \ell^{D\,g}\,(\log q)^{2+\varepsilon} $.
\end{proposition}

\begin{proof}
Postponed to the end of this section.
\end{proof}

Since the geometric resolution requires its input to be a reduced regular
sequence, we first need to ensure that we can construct such a sequence from
our input system. A classical way to achieve this is to replace the input
system by a generic linear
combination of the polynomials. If the ideal generated by the
input system is $0$-dimensional and radical, then a variant of Bertini's
theorem ensures that the obtained sequence is regular and reduced.

\begin{proposition}\cite[Thm.~A.8.7]{SomWam05}
    \label{prop:regred}
Let $(f_1,\ldots, f_m)\in \mathbb
F_q[X_1,\ldots, X_{n_x}, Y_1,\ldots, Y_{n_y}]^m$ be polynomials such that the
ideal $I=\langle f_1,\ldots, f_m\rangle$ has dimension $0$ and is radical. Let
$d_x, d_y$ be two integers such that $\deg_x(f_i)\leq d_x$, $\deg_y(f_i)\leq
d_y$ for all $i\in[1,m]$. Let
$p$ be the characteristic of $\mathbb F_q$, and assume that
$2^{n_x+n_y}d_x^{n_x}\,d_y^{n_y} < p$. For $M$ an $(n_x+n_y)\times m$ matrix with entries in
$\overline{\mathbb F_q}$, let $(f_1^{(M)}, \ldots, f_{n_x+n_y}^{(M)})\in \mathbb
F_q[X_1,\ldots, X_{n_x}, Y_1,\ldots, Y_{n_y}]^{n_x+n_y}$ be defined as
$$\begin{bmatrix}
  f_1^{(M)}\\f_2^{(M)}\\\vdots\\f_{n_x+n_y}^{(M)}\end{bmatrix}
= M\cdot\begin{bmatrix}f_1\\f_2\\\vdots\\f_m\end{bmatrix}.$$
Then there exists a nonempty open subset $\mathcal O\subset \overline{\mathbb
F_q}^{ (n_x+n_y)\times m}$ of the space of $(n_x+n_y)\times m$ matrices such
that for any $M\in\mathcal O$, for any $s\in[1,n_x+n_y]$, and at any
point $(\mathbf x,\mathbf y)\in\overline{\mathbb F_q}^{n_x+n_y}$ such that
$f_1^{(M)}(\mathbf x,\mathbf y) =
\dots = f_s^{(M)}(\mathbf x,\mathbf y) = 0$, the derivatives
$Df_1^{(M)}(\mathbf x,\mathbf y),
\ldots, Df_s^{(M)}(\mathbf x,\mathbf y)$ are linearly independent over
$\overline{\mathbb F_q}$.  In particular, for any $M\in\mathcal O$, the
sequence $(f_1^{(M)},\ldots, f_{n_x+n_y}^{(M)})$ is reduced and regular.
\end{proposition}

\begin{proof} This is a reformulation of \cite[Thm.~A.8.7]{SomWam05} in the
  case of finite fields. In \cite[Thm.~A.8.7]{SomWam05}, this result is stated
  over the field $\mathbb C$, but this statement holds true over any field
  $k$, provided that an extra separability assumption is satisfied. More precisely,
  set $n =n_x+n_y$ and let $V_s\subset \overline{k}^n\times\overline{k}^{n\, m}$
  be the variety of pairs $((\mathbf x,\mathbf y), M)$ such that
  $f_1^{(M)}(\mathbf x, \mathbf y)
  = \dots = f_s^{(M)}(\mathbf x, \mathbf y) = 0$. In this setting, the extra condition that
  is required for the proposition to hold is that the projection $\pi$ of $V_s$ to $\overline{k}^{n\, m}$ must be
  separable for all $s\in[1,n]$ (this is always true in characteristic 0). We refer to
  \cite[Thm.~4.2]{kleiman1997bertini} for more details on this separability
  argument. In our setting, the degree of a generic fiber
  of $\pi$ is bounded by $2^{n}d_x^{n_x}\, d_y^{n_y} < p$ using the multi-homogeneous
  Bézout bound (see e.g. Proposition~\ref{prop:degreebound} below) and hence the separability condition is satisfied.
  \end{proof}

Since we are looking at polynomial systems over finite fields, we must estimate
the size of the extension of the base field that is required to find with
sufficiently large probability a matrix $M$ such that $f_1^{(M)},\ldots,
f_{n_x+n_y}^{(M)}$ is reduced and regular.

\begin{lemma}\label{lem:extneeded}
  Let $(f_1,\ldots, f_m)\in \mathbb F_q[X_1,\ldots, X_{n_x}, Y_1,\ldots,
  Y_{n_y}]^m$ be polynomials satisfying the assumptions of
Proposition~\ref{prop:regred} and such that their total degree is bounded above by
$d\in\N$. Set $n = n_x+n_y$ and  
$$e = \left\lceil (2n+1)\log_q(d+1) + \log_q(11)\right\rceil.$$ If $M$ is an $n\times m$ matrix with entries in
$\mathbb F_{q^e}$ picked uniformly at random, then the probability that
$(f_1^{(M)}, \ldots, f_n^{(M)})$ is a reduced regular sequence is bounded below
by $10/11$.
\end{lemma}

\begin{proof}
Let $\Lambda$ denote an $n\times m$ matrix with indeterminate entries
$$\Lambda = \begin{bmatrix}
  \lambda_{11} & \dots & \lambda_{1m}\\
  \vdots&\vdots&\vdots\\
  \lambda_{n1} & \dots & \lambda_{nm}
\end{bmatrix}$$
and let $F_1(\Lambda, X, Y), \ldots, F_n(\Lambda, X, Y)\in\F_q[X_1,\ldots, X_{n_x},
Y_1,\ldots, Y_{n_y},
\lambda_{11},\ldots, \lambda_{nm}]$ be the polynomials defined as
$$\begin{bmatrix}
  F_1(\Lambda, X, Y)\\
  \vdots\\
  F_n(\Lambda, X, Y)
\end{bmatrix} = \Lambda\cdot 
\begin{bmatrix}
  f_1(X, Y)\\
  \vdots\\
  f_m(X, Y)
\end{bmatrix}.$$

For $s\in[1, n]$, we consider the $s\times m$ matrix $\Lambda^{(s)}$ obtained by
truncating $\Lambda$ to its $s$ first rows, a new set of variables $\{\mu_1, \ldots,
  \mu_{s-1}\}$ and the following polynomial system:
$$
\begin{array}{c}
  F_1(\Lambda^{(s)}, X, Y) = \dots  = F_s(\Lambda^{(s)}, X, Y) = 0 \\[3mm]
  \begin{bmatrix} \mu_1&\cdots&\mu_{s-1}& 1\end{bmatrix} \cdot
\begin{bmatrix}
  \dfrac{\partial F_1}{\partial X_1} & \cdots & \dfrac{\partial
  F_1}{\partial X_{n_x}} & \dfrac{\partial F_1}{\partial Y_1} & \cdots & \dfrac{\partial
    F_1}{\partial Y_{n_y}} \\
  \vdots&\vdots&\vdots&\vdots&\vdots&\vdots\\
  \dfrac{\partial F_s}{\partial X_1} & \cdots & \dfrac{\partial
    F_s}{\partial X_{n_x}}&
  \dfrac{\partial F_s}{\partial Y_1} & \cdots & \dfrac{\partial
    F_s}{\partial Y_{n_y}}
\end{bmatrix} =
\begin{bmatrix} 0 & \cdots & 0\end{bmatrix}
\end{array}$$

This is a system of $n+s$ polynomials of degree bounded above by $d+1$ in
$n+s-1+ms$ variables. By Bézout inequality (see e.g.
\cite[Thm.~1]{heintz1983definability}), this system defines a variety $V_s$ which is either empty, or its
degree is at most $(d+1)^{n+s}$.
We remark that if $V_s$ is not empty, then it has dimension at least $ms -
1$ since its vanishing ideal is generated by $n+s$ elements.
The Zariski closure of its projection $W_s$ to the space $\overline{\mathbb F_q}^{sm}$ of matrices
$\Lambda^{(s)}$ is either empty, the whole space or a proper sub-variety. By
Proposition~\ref{prop:regred}, it must be empty or a proper sub-variety. Next, we remark that
the degree of the image of a variety by a linear projection cannot increase.
Therefore, the sum of the degrees of the irreducible components of $W_s$ is also
bounded by $(d+1)^{n+s}$ if $W_s\neq \emptyset$. In the sequel, we let
$h_s(\lambda_{11},\ldots, \lambda_{sm})$ denote a polynomial vanishing
on $W_s$ of degree bounded by $(d+1)^{n+s}$ (we set $h_s(\lambda_{11},\ldots, \lambda_{sm}) =
1$ if $W_s = \emptyset$).

The Schwarz-Zippel Lemma implies that the cardinality of the set
$$E = \left\{ 
  \begin{bmatrix} 
    M_{11}&\cdots&M_{1m}\\
\vdots&\vdots&\vdots\\
M_{n1}&\cdots&M_{nm}\end{bmatrix}\in\mathbb F_{q^e}^{nm} \mid h_1(M_{11},\ldots, M_{1m})\cdots h_n(M_{11},\ldots, M_{nm})\ne 0\right\}$$
is bounded above by $q^e/11$, for the value of $e$ given in the statement.

The proof is concluded by noticing that for any $M\in E$, for any $s\in[1,n]$, and for any $(\mathbf x, \mathbf y)\in\overline{\mathbb F_q}^n$ such that
$f_1^{(M)}(\mathbf x,\mathbf y)=
\cdots =
f_s^{(M)}(\mathbf x,\mathbf y) = 0$ the derivatives $Df_1^{M}(\mathbf x,\mathbf
y), \ldots,$ $
Df_s^{(M)}(\mathbf x,\mathbf y)$ span
the normal space at $(\mathbf x, \mathbf y)$ to the variety associated with $\langle f_1^{(M)}, \ldots,
f_s^{(M)}\rangle$. Hence, $f_1^{(M)},\ldots, f_n^{(M)}$ is a reduced regular
sequence.
\end{proof}

Once we have a reduced regular sequence, we can use
\cite[Thm.~4.8]{cafure2006fast} to solve the system.  We note that in
\cite{cafure2006fast} there is a general assumption that for all
$s\in[1,n]$ the intermediate ideals $\langle f_1^{(M)},\ldots,
f_s^{(M)}\rangle$ define absolutely irreducible varieties. However, the proof
of \cite[Thm.~4.8]{cafure2006fast} does not require this assumption (this
assumption is only required in algorithms for finding a rational point in
\cite[Section~6]{cafure2006fast}).

Next, we describe the
data structures used in \cite{cafure2006fast} to represent polynomial systems.
The
algorithms take as input polynomials represented by
division-free straight-line programs (DFSLP). A DFSLP defined over a field $k$
is a sequence of polynomials
$h_1, h_2,\ldots, h_\ell\in k[X_1,\ldots,X_n]$ such that each polynomial
$h_i$ is
either a variable $X_t$ with $t\in[1,n]$, an element in $k$, or
$h_i = h_j\circ h_{j'}$, where
$j, j' < i$ and $\circ\in\{+,-, \times\}$ is an arithmetic operation.
The time of a DFSLP is
the total number of arithmetic operations, and its space is the minimal
number of arithmetic registers required to evaluate it. A polynomial system
$f_1,\ldots, f_m$ is said to be represented by a DFSLP $h_1,\ldots,
h_\ell$ if $\{f_1,\ldots, f_m\}\subset\{h_1,\ldots,h_\ell\}$.

\begin{theorem}\cite[Thm.~4.8]{cafure2006fast}
\label{thm:CafMat}
  Let $f_1^{(M)},\ldots, f_n^{(M)}\in\mathbb F_{q^e}[X_1,\ldots, X_n]$ be a
  reduced regular sequence, where the polynomials are represented by a DFSLP 
  with space $\mathcal S'$ and time $\mathcal T'$. Set the following notation:
  \begin{itemize}
    \item The integer $d$ is $\max_{i\in[1,n]}(\deg(f_i^{(M)}))$;
    \item For any real number $x\geq \exp(1)$, $\mathcal U(x) = x(\log x)^2\log\log x$;
    \item Let $\delta\in\N$ be an integer larger than the degrees of the ideals $\langle f_1^{(M)}\rangle, \langle f_1^{(M)},
      f_2^{(M)}\rangle, \ldots,$ $\langle f_1^{(M)},\ldots, f_n^{(M)}\rangle$.
  \end{itemize}
  Assume further that
  $q^e\geq 60\,n^4 d \delta^4$.
  There is a
  probabilistic Turing machine using space $O( (\mathcal
  S'+n+d)\delta^2\log(q^e\delta))$ and time $O( (n\mathcal T' + n^5)\mathcal
  U(\delta)(\mathcal U(d\delta) + \log(q^e\delta))\mathcal U(\log(q^e\delta)))$
  which takes such polynomial systems as input and which outputs a $\mathbb
  F_{q^e}$-geometric resolution of the algebraic set
  $\{\mathbf x\in\overline{\mathbb F_{q^e}}^n\mid f_1^{(M)}(\mathbf x) = \dots
  = f_n^{(M)}(\mathbf x) = 0\}$ with probability at least $11/12$.
\end{theorem}

The next lemma is a first step for preparing our system in order to use
Theorem~\ref{thm:CafMat} for bi-homogeneous systems: we need to estimate the
size and space needed to represent a bi-homogeneous system by a DFSLP.

\begin{lemma}\label{lem:SLPbihom}
Let $d_x, d_y\in\Z_{>0}$ be two positive integers.
  A polynomial system $f_1,\ldots, f_m\in\mathbb F_q[X_1,\ldots,
  X_{n_x},Y_1,\ldots, Y_{n_y}]$ such
  that for all $i\in[1,m]$, $\deg_x(f_i)\leq d_x$ and $\deg_y(f_i)\leq
d_y$ can be represented by a DFSLP with time and space
$O\left((d_x+d_y+m)\binom{n_x+d_x}{n_x}\binom{n_y+d_y}{n_y}\right)$.
\end{lemma}

\begin{proof}
  There are $\binom{n_x+d_x}{n_x}\binom{n_y+d_y}{n_y}$ monomials $\mu$ in
  $\mathbb F_q[X_1,\ldots,X_{n_x},Y_1,\ldots, Y_{n_y}]$ such that
    $\deg_x(\mu)$ $\leq d_x$ and
      $\deg_y(\mu)\leq d_y$. We consider the DFSLP which starts by evaluating
      these monomials. This costs less than
      $\binom{n_x+d_x}{n_x}\binom{n_y+d_y}{n_y}(d_x+d_y-1)$ multiplications,
      using a naive algorithm. Then we multiply each of these monomials by the
      corresponding coefficients, and we sum. This costs
      $m\binom{n_x+d_x}{n_x}\binom{n_y+d_y}{n_y}$ multiplications and $m(\binom{n_x+d_x}{n_x}\binom{n_y+d_y}{n_y}-1)$ additions.
\end{proof}

The next ingredient in order to derive Proposition~\ref{prop:mainpolsys} from
Theorem~\ref{thm:CafMat} is an upper bound on~$\delta$. This can be obtained 
via the multi-homogeneous Bézout bound. 

\begin{proposition}\label{prop:degreebound}
  Let $f_1,\ldots, f_m$ be a regular sequence in $\mathbb F_q[X_1,\ldots,
  X_{n_x}, Y_1,\ldots, Y_{n_y}]$ and $d_x,d_y\in\N$ be such
  that for any $i\in[1, m]$, $\deg_x(f_i)
  \leq d_x$ and $\deg_y(f_i)\leq d_y$. Then the degree of the ideal $\langle
  f_1,\ldots, f_m\rangle$ is at most
  \begin{equation}\sum_{\substack{j_1+j_2 = m\\0\leq j_1\leq n_x\\0\leq j_2\leq n_y}}
    \binom{m}{j_1} d_x^{j_1} d_y^{j_2}.\label{eq:degreebound}\end{equation}
  Moreover, this degree is bounded above by $2^{n_x+n_y} d_x^{n_x} d_y^{n_y}$.
\end{proposition}

\begin{proof}
This is a direct consequence of \cite[Prop. I.1]{SafSch17} using, with the
notation of \cite[Prop. I.1]{SafSch17}, $k=1$, $e=0$, $P=m$, $D_{i,0} = d_x$, $D_{i,1} =
d_y$, $n=n_x$, $n_1=n_y$.
Note that \cite[Prop. I.1]{SafSch17} is stated when the base field is
$\mathbb C$, but the proof works without any major modification when the base
field is a finite field.
The last sentence of the statement follows from the fact that the regularity
assumption implies that $m\leq n_x+n_y$, and hence the sum of the binomial
coefficients is bounded above by $2^m\leq 2^{n_x+n_y}$.
\end{proof}

We now have all the ingredients needed to prove Proposition~\ref{prop:mainpolsys}.

\begin{proof}[Proof of Proposition~\ref{prop:mainpolsys}]
  Set $n=n_x+n_y$. First, we note that if $f_1,\ldots, f_m\in\mathbb F_q[X_1,\ldots, X_{n_x},$ $
  Y_1,\ldots, Y_{n_y}]$ is represented by a straight-line
program over $\mathbb F_q$ with space $\mathcal S$ and time $\mathcal T$, then
for any $e\in\N$ and any $m\times n$
matrix $M$ with entries in $\mathbb F_{q^e}$, the sequence $f_1^{(M)},\ldots,
f_n^{(M)}\in\mathbb F_{q^e}[X_1,\ldots, X_{n_x}, Y_1,\ldots, Y_{n_y}]$ can be represented by a
straight-line program over $\mathbb F_{q^e}$ with space $\mathcal S'$ and time
$\mathcal T'$, where $S' = O(\mathcal S)$ and $\mathcal T'=O(\mathcal
T+m\,n)$. 
We consider the probabilistic Turing machine which performs the following steps:
\begin{enumerate}
\item \label{it:TM1}It chooses an $m\times n$ matrix uniformly at random with entries in
  $\mathbb F_{q^e}$, with 
  $$e = \max\left(\left\lceil (2n+1)\log_q(d+1) +
  \log_q(11)\right\rceil,\left\lceil \log_q(60\, n^4\, d\,
  \delta)\right\rceil\right), $$
  where $d=d_x+d_y=(\ell^C+1)\,h(g)$, $n= n_x+n_y = C g + h(g)$, $\delta = 2^n
  d_x^{n_x} d_y^{n_y} = (2 h(g))^{C g+h(g)}\ell^{C^2 g}$. Using the
  inequalities $n_x < C\,g, n_y < h(g), d_x < h(g)\ell^C, d_y < h(g)$, we get
  that $e=O_g(\log_q\ell)$;
\item It constructs the straight-line program representing $f_1^{(M)},\ldots,
  f_n^{(M)}$ with space $\mathcal S'=O(\mathcal S)$ and time $\mathcal
  T'=O(\mathcal T+m\,n)$;\label{it:TM2}
\item It applies the probabilistic Turing machine from
  Theorem~\ref{thm:CafMat} to compute a geometric resolution of the algebraic
  set defined by $f_1^{(M)}(X) = \dots = f_n^{(M)}(X) = 0$; By
  Theorem~\ref{thm:CafMat}, it returns a geometric resolution $((\ell_1,\ldots,
  \ell_n), q(T),(
  q_1(T),\ldots, q_n(T)))$ provided that $f_1^{(M)}(X),\ldots,f_n^{(M)}(X)$ is a
  reduced regular sequence;\label{it:TM3}
\item It computes $\lambda(T) = {\sf GCD}(q(T), f_1(q_1(T),\ldots, q_n(T)), \ldots,
f_m(q_1(T),\ldots, q_n(T)))$;\label{it:TM4}
\item It computes $\nu_1(T) = q_1(T)\bmod \lambda(T),
  \ldots,
  \nu_n(T) = q_n(T)\bmod\lambda(T)$ and returns the geometric resolution
  $((\ell_1,\ldots, \ell_n), \lambda(T), (\nu_1(T), \ldots, \nu_n(T)))$.\label{it:TM5}
\end{enumerate}

We start by showing that the output of this algorithm is indeed a geometric
resolution of the algebraic set $V = \{\mathbf x\in \overline{\mathbb F_q}^n\mid
f_1(\mathbf x) = \dots = f_m(\mathbf x) = 0\}$, assuming that the probabilistic algorithm in
Step~\ref{it:TM3} returns the correct result. Let $W$ be the algebraic set
$\{\mathbf x\in \overline{\mathbb F_q}^n\mid f_1^{(M)}(\mathbf x) = \dots =
f_m^{(M)}(\mathbf x) = 0\}$. Since $\langle f_1^{(M)},\ldots, f_n^{(M)}\rangle
\subset \langle f_1,\ldots, f_m\rangle$,  we have $V\subset W$. 
By construction, the algebraic set defined by the geometric resolution
  $((\ell_1,\ldots, \ell_n), \lambda(T), (\nu_1(T), \ldots, \nu_n(T)))$
is precisely the subset of $W$ where all
polynomials $f_1,\ldots, f_m$ simultaneously vanish.

It remains to prove that this Turing machine runs within the desired complexity.
Steps~\ref{it:TM1} and \ref{it:TM2} require negligible time. Step~\ref{it:TM3} is done within space $O( (\mathcal
  S'+n+d)\delta^2\log(q^e\delta))$ and time $\widetilde O( (n\mathcal T' + n^5)\delta(d\delta + \log(q^e\delta))\log(q^e\delta))$
  (Theorem~\ref{thm:CafMat}), provided that $\delta$ is an upper bound on the
  degrees of the intermediate ideals. Step~\ref{it:TM4} is done within space and time
  bounded by $\widetilde O(\delta\,e\log q (\mathcal T + m))$ by evaluating the
  SLP modulo $q(T)$ (whose degree is bounded by $\delta$) and then by computing
  $m$ GCD using a quasi-linear algorithm.
  Finally, Step~\ref{it:TM5} can be done within time and space $\widetilde
  O(\delta\,e\log q)$.

Then, Proposition~\ref{prop:degreebound} shows
that $\delta$ is an upper bound on the degrees of the intermediate ideals.
Using the facts that $\binom{n_x+d_x}{d_x}\leq(n_x+d_x)^{n_x} = O_g(\ell^{C^2 g})$
and $\binom{n_y+d_y}{d_y}\leq (n_y+d_y)^{n_y} = O_g(1)$, Lemma~\ref{lem:SLPbihom} provide bounds on $\mathcal S$ and
$\mathcal T$.
Summing the complexities leads to the claimed
complexity estimate.
Finally, the probability of success is bounded below by the probability that
the sequence $f_1^{(M)}, \ldots, f_n^{(M)}$ is reduced and regular
(Lemma~\ref{lem:extneeded}) multiplied
by the probability of success of the probabilistic Turing machine in
Theorem~\ref{thm:CafMat}, namely $10/11\cdot 11/12 = 5/6$.
\end{proof}

\section{Computing generic $\ell$-torsion points}\label{sec:generic}

Let $\C$ be a hyperelliptic curve of genus $g$ over $\F_q$ with Weierstrass
form $Y^2=f(X)$ ($f$ monic, squarefree, and $\deg(f)=2g+1$) and $J$ be its
Jacobian. Let $\ell>g$ be a prime not dividing $q$. In this section, we define
a notion of genericity for $\ell$-torsion elements in $J$ and we show that a
geometric resolution for the variety they form can be computed efficiently
using the tools described in Section~\ref{sec:polsys} by solving a polynomial
system of $g^2+g$ equations in $g^2+g$ variables.  Our starting point is the
modelling of the $\ell$-torsion sketched by Cantor in the point (5) of
Section~9 of~\cite{Ca94}. This section and the next one that deals with the
non-generic cases rely heavily on the Mumford representation, that we recall
here, and refer to~\cite{GalbraithMPKC} for more details.

\begin{definition}[Mumford representation]
    Every element in $J$ can be uniquely represented by a pair of
    polynomials $\langle u(X), v(X)\rangle$, where $u$ is monic of degree
    $w\le g$, the polynomial $v$ has degree less than $w$, and $u$
    divides $v^2-f$. 
\end{definition}

We call the integer $w$ in this definition the weight of the element.  An
element in Mumford representation $\langle u(X), v(X)\rangle$,
corresponds to a divisor (called reduced divisor) of the form $\sum_{1\le
i\le w} (P_i - \infty)$, where the $P_i=(x_i,y_i)$ are the affine points
of $\C$ such that $u(x_i)=0$ and $y_i=v(x_i)$, with appropriate
multiplicities. This implies in particular that if two $P_i$'s share the
same $x$-coordinate, then they are equal.

In what follows, we often also call Mumford representation a pair of
polynomials where $u$ is not monic. In that case, unicity of the
representation is no longer guaranteed, but there is no ambiguity in the
element of $J$ represented this way.
\medskip

In genus 1, the $\ell$-torsion points are the points whose abscissae are
the roots of the $\ell$-division polynomial, which has degree
$O(\ell^2)$. For higher genera, Cantor~\cite{Ca94} described analogous polynomials
$\delta_{\ell}$ and $\varepsilon_{\ell}$ such that, for $(x,y)$ a generic
point of the curve and $\ell > g$, we have
\begin{equation*}
    \ell\cdot\left(
    (x,y)-\infty\right)=\left\langle\delta_{\ell}\left(\frac{x-X}{4y^2}
    \right),\varepsilon_{\ell}\left(\frac{x-X}{4y^2}
    \right)\right\rangle.
\end{equation*}

\begin{lemma}\label{lem:deg_delta}
    The polynomial $\delta_{\ell}(X)$ has degree $g$ and its coefficients
    are polynomials in $\F_q[x]$ of degree bounded by
    $\frac{1}{3}g\ell^3 + \Og(\ell^2)$.
    The polynomial $\varepsilon_{\ell}(X)/y$ has degree less than $g$
    and its coefficients are rational fractions in $\F_q(x)$. The degrees
    of the numerators and denominators of these coefficients are bounded
    by $\frac{2}{3}g\ell^3 + \Og(\ell^2)$. Furthermore, any root of a
    denominator is also a root of the leading coefficient of
    $\delta_\ell(X)$.
\end{lemma}

\begin{proof}
    An exact formula for the degree of the leading coefficient of
    $\delta_\ell(X)$ was given
    by Cantor. For the other coefficients and the claims about the roots
    of the denominator,
    the proof is postponed to Section \ref{sec:proof}. We remark that our
    bounds are not tight but they are sufficient for our purpose.
\end{proof} 

We shall prove that Lemma~\ref{lem:deg_delta} is also valid for a non-prime
integer $\ell$. This will be useful in Section~\ref{sec:nongeneric} where
we handle non-generic situations.
\medskip

Later on, we will need explicit names for these coefficients of
$\delta_\ell$ and $\varepsilon_\ell$, so we define the univariate
polynomials $d_i$ and $e_i$ (the notation
does not show the dependence on $\ell$ for simplicity)
such that, after clearing denominators we have:
$$ \delta_{\ell}\left(\frac{x-X}{4y^2}\right) = \sum_{i=0}^gd_i(x)X^i,
\quad \text{and}\quad
\varepsilon_{\ell}\left(\frac{x-X}{4y^2}\right) =
y \sum_{i=0}^{g-1} \frac{e_i(x)}{e_g(x)}X^i.$$

\begin{definition}\label{def:divgen}
    In what follows, we shall say that an element of $J$ is 
    $\ell$-generic if it has weight $g$ and the corresponding
    reduced divisor $\sum_{i=1}^g (P_i-\infty)$ satisfies the following two
    properties:
    \begin{itemize}
        \item For any $i$, the $u$-coordinate of the divisor
            $\ell\cdot(P_i-\infty)$ in Mumford form has   
	    degree $g$;
        \item For any $i\ne j$, the $u$-coordinates of the
            divisors $\ell\cdot(P_i-\infty)$ and $\ell\cdot(P_j-\infty)$
            are coprime.
    \end{itemize}
    
    This implies that all the $P_i$ are distinct. 
    This also implies that if an affine point $P$ occurs in the support of a
    $\ell\cdot(P_i-\infty)$ then neither $P$ nor $-P$ appears in the support
    of another $\ell\cdot(P_j-\infty)$.
\end{definition}

\begin{proposition}\label{prop:syshyp}
    For any $\varepsilon>0$, there is a constant $D$ 
    such that for all
    prime $\ell >g$ coprime to the base field characteristic, there is a
    Monte Carlo algorithm which computes an $\mathbb
    F_{q^e}$-geometric
    resolution of the sub-variety of $J[\ell]$ consisting of
    $\ell$-generic $\ell$-torsion elements, where $e=O_g(\log\ell)$. 
    The time and space complexities of this algorithm are bounded by
    $O_g(\ell^{Dg}(\log q)^{2+\varepsilon})$
    and it returns the correct result with probability at least $5/6$.
\end{proposition}
\begin{proof}
  Let $D=\sum_{i=1}^g(P_i-\infty)$ be an $\ell$-generic divisor
  in $J$. We shall consider a system equivalent to $\ell\cdot D=0$ but let
  us first introduce some notation. For each point $P_i=(x_i,y_i)$ in
  the support of $D$, we denote $\langle u_i,v_i \rangle$ the Mumford
  form of $\ell\cdot(P_i-\infty)$ and $(\alpha_{ij},\beta_{ij})_{1\leq j\leq
  g}$ the
  coordinates of the $g$ points in its support counted with
  multiplicities, which means that for
  any $i$ the $g$ roots of $u_i$ are exactly the $\alpha_{ij}$, and
  that for any $j$, $\beta_{ij}=v_i(\alpha_{ij})$. Note that using
  the previous notation, $u_i(X)=\delta_{\ell}\left(\frac{x_i-X}{4y_i^2}
  \right)$ and $v_i(X)=\varepsilon_{\ell}\left(\frac{x_i-X}{4y_i^2}
  \right)$. 

  We have $\ell\cdot D=0$ if and only if the sum of the divisors
  $\sum_{i=1}^g
  \ell\cdot (P_i-\infty)$ is a principal divisor. The only pole is at
  infinity, so this is equivalent to the existence of a non-zero
  function $\varphi\in\mathbb F_q(\C)$ of the form $P(X)+YQ(X)$ with $P$ and $Q$ two
  polynomials such that the $g^2$ points $(\alpha_{ij},\beta_{ij})$ are
  the zeros of $\varphi$, with multiplicities. Since we want $\varphi$ to 
  have $g^2$ affine points of intersection with the curve $\mathcal{C}$ 
  (once again, counted with multiplicities), the polynomial
  $\Res_Y(Y^2-f,P+YQ)=P^2-fQ^2$ must have degree $g^2$ which yields
  $2\deg(P)\le g^2$ and $2\deg(Q)\le g^2-2g-1$. Exactly one of those
  two bounds is even (it depends on the parity of $g$), and for
  this particular bound, the inequality must be an equality, otherwise
  the degree of the resultant would not be $g^2$. Since the function
  $\varphi$ is defined up to a multiplicative constant, we can
  normalize it so that the polynomial $P^2+fQ^2$ is monic, which is
  equivalent to enforce that either $P$ or $Q$ is monic depending on the
  parity of $g$.

  For a fixed $i\in[1,g]$, requiring the $(\alpha_{ij},\beta_{ij})$ to be zeros of
  $\varphi$ amounts to asking for the $\alpha_{ij}$ to be roots of
  $P(X)+Q(X)v_i(X)$, with multiplicities. Since the $\alpha_{ij}$ are by definition the
  roots of the $u_i$, $\ell\cdot D=0$ is equivalent to $g$ congruence
  relations $P+Qv_i \equiv 0 \bmod u_i$ which we can rephrase using
  Cantor's polynomials:
  \begin{equation}\label{eq:moddelta}
    P(X)+\varepsilon_{\ell}\left(\frac{x_i-X}{4y_i^2} \right)Q(X)
    \equiv 0 \bmod \delta_{\ell}\left(\frac{x_i-X}{4y_i^2} \right).
  \end{equation}
  Thus, for any $\ell$-generic divisor, $\ell\cdot D=0$ is equivalent to the
  existence of $P$ and $Q$ satisfying the above $g$ congruence relations.

  The variables are the coefficients of $P$ and $Q$, as well as
  the $x_i$ and $y_i$. With the degree conditions and the
  normalization, we have $g^2-g$ variables coming from $P$ and $Q$.
  Adding the $2g$ variables $x_i$ and $y_i$, we get a total of
  $g^2+g$ variables.
  Each one of the $g$ congruence relations~\eqref{eq:moddelta} amounts to $g$
  equations providing a total of $g^2$ conditions on
  the coefficients of $P$ and $Q$. The fact that the $(x_i,y_i)$ are
  points of the curve yields the $g$ additional equations $y_i^2=f(x_i)$.
  Finally, we have to enforce the $\ell$-genericity of the solutions,
  which can be done by requiring that $\prod_i d_g(x_i)
  \prod_{i<j}\Res(u_i,u_j)\not=0$.  Therefore, we get a polynomial
  system with $g^2+g$ equations in $g^2+g$ variables, together with an
  inequality. We remark that in principle, the denominators $e_g(x_i)$
  involved in $\varepsilon_\ell$ would generate additional 
  conditions, but by Lemma~\ref{lem:deg_delta} this is already covered
  by the condition $d_g(x_i)\ne 0$.

  In order to apply Proposition~\ref{prop:mainpolsys}, we now estimate
  the degrees to which the variables occur in the equations. We start
  with the equations coming from~\eqref{eq:moddelta}. Each congruence
  relation is obtained by reducing $P+Qv_i$, which is a polynomial of
  degree $O(g^2)$ in $X$, by $u_i$ which is of degree $g$.
  We can do it by repeatedly replacing
  $X^g$ by $-\sum_{j<g} (d_j(x_i)/d_g(x_i)) X^j$,
  which we will have to do at most
  $O(g^2)$ times. Since by Lemma~\ref{lem:deg_delta} the $d_j$ have degree 
  in $O_g(\ell^3)$ in $x_i$ the fully reduced polynomial will have
  coefficients that are fractions for which
  the degrees of the numerators and of the
  denominators are at most
  $O_g(\ell^3)$ in the $x_i$ variables. In these equations, the degree
  in the $y_i$ variables and in the variables for the coefficients of
  $P$ and $Q$ is 1. The degrees in $x_i$ and $y_i$ in the curve
  equations are $2g+1$ and $2$ respectively.

  It remains to study the degree of the inequality. Each resultant is
  the determinant of a $2g\times 2g$ Sylvester matrix whose coefficients are the
  $d_i$, which have degrees bounded by $O_g(\ell^3)$. Since for any $i$
  there are exactly $g$ resultants involving $x_i$ in the product, the
  degree of this inequality in any $x_i$ is in $O_g(\ell^3)$, and it
  does not involve the other variables. In order to be able to use
  Proposition~\ref{prop:mainpolsys}, we must model this inequality by an
  equation, which is done classically by introducing a new variable $T$ and
  by using the equation $T\cdot\prod_i d_g(x_i)
  \prod_{i<j}\Res(u_i,u_j)=1$.

  To conclude, we have a polynomial system with two blocks of
  variables: the $2g$ variables $x_i$ and $y_i$ and the $g^2-g$
  variables coming from the coefficients of $P$ and $Q$. The degree of
  the equations in the first block of variables grows cubically in
  $\ell$, while the degree in the other block of variables depends only
  on $g$. The system therefore verifies the conditions of
  Proposition~\ref{prop:mainpolsys} and the complexity follows, provided
  that we can show that the system is $0$-dimensional and radical.

  Let us consider the sub-variety $S\subset J[\ell]$ consisting of
  $\ell$-generic $\ell$-torsion elements, and $I$ the corresponding
  ideal. More precisely, we see $I$ as the ideal of a sub-scheme of the
  $\ell$-torsion scheme, which is the kernel of a finite and étale map
  because $\ell$ is coprime to the characteristic.
  Therefore $I$ is 0-dimensional and radical. Since all the elements in
  $S$ have the same weight~$g$ we can use the Mumford coordinates
  $\langle u(X), v(X)\rangle$ with $\deg u=g$ and $\deg v<g-1$ as a
  local system of coordinates to represent them. But the polynomial
  system that we have built is with the $(x_i,y_i)$ coordinates, that
  is, it generates the ideal
  $I^\mathrm{unsym}$ obtained by adjoining to the equations defining
  $I$ the $2g$ equations coming from $u(X) = \prod(X-x_i)$ and
  $y_i=v(x_i)$. Then we have $\deg I^\mathrm{unsym} = g! \deg I$.
  By the $\ell$-genericity condition, all the fibers in the variety
  have exactly $g!$ distinct points corresponding to permuting the
  $(x_i,y_i)$ which are all distinct. Therefore the radicality of $I$
  implies the radicality of $I^\mathrm{unsym}$ and we can apply
  Proposition~\ref{prop:mainpolsys} to our polynomial system.
\end{proof}

We emphasize that, although the algorithm in Proposition~\ref{prop:syshyp} is
Monte Carlo, we expect that it returns a correct and verifiable result in most
of the cases. Indeed, if all the $\ell^{2g}-1$ nonzero $\ell$-torsion elements are
$\ell$-generic (which is the situation that we expect to happen in most of the
cases) and if the algorithm returns the correct result, then we can check that
these elements are indeed $\ell$-torsion elements, and that we have all of
them. In that favorable case, the proof of Proposition~\ref{prop:Jl} is
completed.

\section{Non-generic cases}\label{sec:nongeneric}

For most of the curves, we expect that for all the primes $\ell$
considered in Algorithm~\ref{algo:pila} the set $J[\ell]$ contains only
$\ell$-generic elements (apart from 0), so that the result of the
previous section is sufficient. If this is not the case, then it is very
likely that the orbit under the Frobenius endomorphism of the $\ell$-torsion elements computed contains an
$\F_\ell$-basis of $J[\ell]$, so that we can easily recover the missing
elements using the group law or the Frobenius. Still, unless we could prove otherwise, we
can not exclude the case where the set of $\ell$-generic $\ell$-torsion
elements generate a proper subgroup of $J[\ell]$ which is stable under the action of $\varphi$. 
In that unlikely case,
we would maybe not be able to deduce $\chi_\ell$. An option is then to
skip this unlucky $\ell$ and proceed with the algorithm; this would only
marginally increase the largest considered $\ell$. But then, we would be
left to prove that the number of unlucky $\ell$'s is small enough.
The number of isomorphism classes of hyperelliptic curves of genus $g$
over $\F_q$ grows like $q^{2g-1}$. Assuming the non-genericity of
the $\ell$-torsion is driven by the vanishing of some polynomial of
degree $\ell^{O(1)}$ that we consider as random,
independent events, we see that when $q$ grows to infinity and $g$ is
fixed, the probability of having a single curve for which the $\ell$-torsion is
non-generic for all the small $\ell$'s goes to zero. But this kind of
discussion does not seem to lead to a provable result, because the
polynomials involved have a highly non-generic Galois structure, and
furthermore because there exist curves for which this can even become more
structured, for instance those whose Jacobians are not simple.

Our only remaining option is to perform a tedious, systematic study of
all the non-generic cases and to show that they can all be modelled by
polynomial systems that can be solved within the target complexity. The
number of these systems must also be bounded independently of $\ell$, so
that with our setting where $g$ is fixed and $q$ grows to infinity the
global complexity remains the same. All this is the purpose of the
present section. As a warm-up, we will first describe some simple
degeneracy cases and, informally, how to deal with them. Then we will
consider the most general non-generic case and construct the polynomial
system to model it in order to reach our result. Note that all
these systems will have more equations ($O(g^4)$, see Table~\ref{tab:degrees}) than
variables ($O(g^2)$, see Table~\ref{tab:variables}), which is no wonder
since we expect them to have no solution in general.

\subsection{Simple degeneracies}

\paragraph*{Case 1: Low weight $\ell$-torsion elements.}

In order to compute the $\ell$-torsion elements that satisfy all the
conditions of $\ell$-genericity except that their weight is less than $g$, we
can proceed as in the proof of Proposition~\ref{prop:syshyp}
with the following modifications.  This time, $D=\sum_{i=1}^w
(P_i-\infty)$, and the only difference is that there are $w$ points
instead of $g$. Following the same method, we search $\varphi$ of the
form $P(X)+YQ(X)$ such that the points in the reduced divisor
$\ell\cdot(P_i-\infty)$ are exactly the zeros of $\varphi$. We now want
$\varphi$ to have $gw$ points of intersection with $\C$ instead of $g^2$,
and we similarly deduce $2\deg(P)\le gw$ and $2\deg(Q) \le gw-2g-1$.  By
similar parity considerations we deduce that exactly one of these bounds
is even, and the corresponding polynomial will be made monic to normalize
the function.  The number of variables from $P$ and $Q$ is thus $gw-g$,
and after adding the $2w$ variables $x_i$ and $y_i$, we have a total of
$(g+1)w+w-g$ variables. As for the number of equations, the number of
congruence relations is now $w$ but the relations themselves remain
unchanged, and we get a total of $(g+1)w$ equations after adding the $w$
equations $y_i^2=f(x_i)$. Since we keep the degrees unchanged but reduce
the number of variables, the complexity bounds are still valid in this
case.

\paragraph*{Case 2: Multiple points in the $\ell$-torsion divisor.}

It may happen that the reduced forms of $\ell$-torsion divisors contain multiple
points. In that case, the $u$-coordinate in the Mumford representation of such
a point is not squarefree. Although the modelling by the polynomial system
described in Section~\ref{sec:generic} is still faithful, such multiple points
will induce multiplicities since what we actually compute is the variety
describing the points in the reduced divisor. Therefore, the ideal generated 
by the polynomial system is not radical in this case. We use the following
workaround: For $\lambda=(\lambda_1,\ldots, \lambda_k)$ a partition
of $w$,
we write a polynomial system generating a radical ideal whose solutions
represent the reduced divisors of the form $D = \lambda_1\,P_1 + \dots +
\lambda_k\,P_k - w\,\infty$. To build this polynomial system, we do as if we were
looking for elements of weight $k$, but instead of multiplying $P_i$ by
$\ell$, we multiply it by $\lambda_i\,\ell$, using Cantor's polynomials
$\delta_{\lambda_i\ell}$ and $\varepsilon_{\lambda_i\ell}$. This
system has the same number of variables and equations as if we were looking for
elements of weight $k$. Since $\lambda_i$ is bounded above by $g$, the
degrees of the equations are multiplied by a quantity which depends only on $g$
but not on $\ell$. Consequently, the complexity bounds are still valid in this
case. To avoid multiplicity problems that could arise from
subpartitions of $\lambda$, we add the inequalities 
$x_i\neq x_j$ for $i\neq j$, where $x_i$ is the $x$-coordinate of $P_i$.
Again, this does not change our complexity estimate.

\paragraph*{Case 3: Low weight after multiplication by $\ell$.}

We study here the case where the $\ell$-genericity property that is not
verified is that the $\ell\cdot(P_i-\infty)$ are of weight $g$, all the
others being satisfied. We denote by $w_i\le g$ the weight of
$\ell\cdot(P_i-\infty)$. Then each $u_i$ will have degree $w_i$, so that
each congruence relation~\eqref{eq:moddelta} yields only $w_i$ equations
instead of $g$. In Cantor's article (on top of page 141 in~\cite{Ca94}),
it is stated that $\ell\cdot(P_i-\infty)$ is of weight $w_i$ if
and only if for any $k$ such that $w_i< k \le g$ we have
$\psi_{\ell-k+w_i+1}(x_i)=0$ and $\psi_{\ell-g+w_i}(x_i)\not=0$, where the polynomials
$\psi_i$ are efficiently computable and of degrees bounded by $O_g(\ell^2)$.
Therefore the total number of equations is
unchanged.  Since the function $\varphi$ will have to vanish at $\sum_i
w_i$ points instead of $g^2$, we also reduce the degree of $P$ and $Q$
accordingly. The number of variables from $P$ and $Q$ thus becomes
$\sum_iw_i-g$ which is smaller than in the generic case, while the number
of equations remains the same, and their degrees are also smaller.
Thus we can still describe this non-generic situation with
systems that can be handled within the same complexity bounds.

\paragraph*{Case 4: Non semi-reduced principal divisor.}

We now consider the case where the $\ell$-genericity property fails due
to the presence of a point of abscissa $\xi$ which appears with positive
multiplicity $\nu_{i}$ in an $\ell\cdot(P_i-\infty)$ and with a negative
multiplicity $-\nu_{j}$ in another $\ell\cdot(P_j-\infty)$. Let
$\nu=\min(\nu_i,\nu_j)$. This event implies
that $(X-\xi)^\nu$ divides both $P$ and $Q$ so that we can write
$\varphi(X,Y)=(X-\xi)^\nu(\widetilde{P}(X)+Y\widetilde{Q}(X))$,
with $\widetilde{P}$ coprime
to $\widetilde{Q}$.
The number of variables coming from $\varphi$ is reduced compared to the
generic case: we add one (the variable $\xi$), but the number of
coefficients in $\widetilde{P}$ is reduced by $\nu$ compared to $P$, and
the same is true for $\widetilde{Q}$ and $Q$.
To write the conditions on $\varphi$, we write the congruences exactly
like in the generic case and we add conditions to ensure that the
multiplicities are respected. Namely, $u_i$, $u_j$ and $v_i+v_j$ must all
be divisible by $(X-\xi)^\nu$, which adds $3\nu\le 3g$ equations. The
degree in $\xi$ in these equations is bounded by $g^2$. Since this does
not depend on $\ell$, the complexity result is maintained.
The general study will cover the case where there are several $\xi$'s at
which the semi-reduction genericity assumption fails. Also, there is no
reason why such a root $\xi$ should occur in only two of the
$\ell\cdot(P_i-\infty)$'s. Such a situation will be also taken into
account in Section~\ref{sec:combining}.

\paragraph*{Case 5: Multiplicity in $\ell\cdot D$.}

The last situation that could lead to not satisfying $\ell$-genericity is
when the same point is shared within different $\ell\cdot(P_i-\infty)$,
which causes some trouble as the congruence relations of the generic case
will not be able to handle the subsequent multiplicity. Note that if the
multiplicity occurs only within a single $\ell\cdot(P_i-\infty)$ this is
already dealt within the generic case. One can view our method
as using the Chinese remainder theorem on the modular
conditions~\eqref{eq:moddelta} to see that multiplicities within a single
congruence is handled whereas common factors within different
$u_i$-polynomials are an obstacle that needs special strategies.
There is some similarities with the previous case that also implies a
common factor between two different $u_i$'s. 

We devise the following workaround: instead of considering the
congruences modulo the $u_i$'s separately, we group them into a single
congruence of the form $P+QV\equiv 0 \bmod U$, with $U=\prod_i u_i$ and
$V$ a polynomial whose coefficients shall be new variables such that $V
\equiv v_i \bmod u_i$ for all $i$.  Note that if some non semi-reduced
case occurs simultaneously, $U$ must actually be divided by the
aforementioned $X-\xi$; such situations will be dealt with later, in the
general study (Section~\ref{sec:combining}).
In order for $V$ to encode enough information and ensure that the
condition $P+QV \equiv 0 \bmod U$ enforces a function with exactly the
correct principal divisor, we have to follow Mumford's representation and
add the condition $U | V^2 -f$, with $\deg V < \deg U$. Together with
the other conditions on $U$ and $V$, we then have existence and unicity (up to a constant factor): 
they are the result of Cantor's composition algorithm.

In order to write the polynomial system modelling this situation, some
care must be taken so as to stay within the scope of
Proposition~\ref{prop:mainpolsys}. The polynomial $U$ is of degree $g^2$
and its coefficients are polynomials in the $x_i$'s of degrees bounded by
$O_g(\ell^3)$.  New variables are added for the coordinates of $V$. For
each $i$, the condition $V \equiv v_i \bmod u_i$ is converted in $O(g)$
equations, with degrees $O_g(\ell^3)$ in $x_i$ and $1$ in the coordinates
of $V$.  The condition $U | V^2-f$ contributes to $O(g^2)$ additional
equations, each of them of degree $2$ in the coordinates of $V$, and degree
$O_g(\ell^3)$ in the coordinates $x_i$. And finally, the equation
$P+QV\equiv 0 \bmod U$, contributes also to $O(g^2)$ equations, each of
them of degree $1$ in the coordinates of $V$, $P$ and $Q$, and of degree
$O_g(\ell^3)$ in the coordinates $x_i$.
Skipping the details, we can again apply
Proposition~\ref{prop:mainpolsys} and get the expected complexity.

\subsection{Combining all possible degeneracies}
\label{sec:combining}
\paragraph{A data structure to describe each type of non-genericity.}

We want to describe a family of polynomial systems that covers all the
possible non-generic cases, possibly mixing all kind of problems that
have been listed. We begin by grouping together non-genericity situations
that can be covered by the same polynomial system.

We consider an $\ell$-torsion divisor $D$
of weight $w\le
g$ (like in case 1). Next, a partition $\lambda=(\lambda_1,\ldots,
\lambda_k)$ of $w$ is picked to represent the multiplicity pattern in the
$u$-coordinate of the $\ell$-torsion divisor, as in case~2 so that
$D=\sum_{i=1}^k\lambda_i(P_i-\infty)$. Then, a vector
$t=(t_1,\ldots,t_k)$ is chosen, to
represent the weights of the $P_i$ after multiplication by $\lambda_i\,\ell$ as in case~3: For
$i$ in $[1,k]$, the reduced divisor $\lambda_i\,\ell\cdot(P_i-\infty)$ is of weight
$t_i$. Then, we need to consider how many common or opposite points these
divisors are in their support to take into account the cases~4 and~5.  We
denote by $Q_1,\ldots,Q_s$ the points in the union of the supports of all
the reduced divisors $\lambda_i\,\ell\cdot(P_i-\infty)$, keeping only one point in
each orbit under the hyperelliptic involution.
We represent the non-genericity by a $k\times
s$ matrix $M$ such that its non-zero entries $m_{ij}$ verify
$m_{ij}=\ord_{Q_j}(\lambda_i\,\ell\cdot(P_i-\infty))$ when
$Q_j$ is in the support of $\lambda_i\,\ell\cdot(P_i-\infty)$ or
$m_{ij}=-\ord_{Q_j'}(\lambda_i\,\ell\cdot(P_i-\infty))$ when the hyperelliptic
conjugate $Q_j'$ of $Q_j$ is in the support. Note
that this matrix, that we shall call the matrix of shared points,
represents both multiplicities and non-semi-reduction. Since the row $i$
represents what happens with points in the support of
$\lambda_i\,\ell\cdot(P_i-\infty)$, which is of weight $t_i$, the sum of the
absolute values of the entries of the row $i$ of $M$ is equal to $t_i$.

Also, by construction,
in each column, there is at least one non-zero entry.
An additional complication arises when one of the $P_i$ is a ramification
point, i.e. when its $y$-coordinate is zero, because this would cause
multiplicities if care is not taken, leading to non-radicality of the
polynomial system we build. Since this corresponds to $P_i-\infty$ being
of order 2, the weight $t_i$ is equal to $\lambda_i\ell \bmod 2$, namely
0 or 1. If $t_i=0$, then the divisor $D-\lambda_i(P_i-\infty)$ is also an
$\ell$-torsion divisor of weight $w-\lambda_i$, so that we can reconstruct $D$
from another polynomial system. There is however no obvious way to
preclude the possibility $t_i=1$. Therefore, we will encode the fact
that $P_i$ is a ramification point by a bit $\epsilon_i$ that can be set
only in the cases where $t_i=1$ and $\lambda_i=1$.

A tuple $(w, \lambda=(\lambda_1,\ldots, \lambda_k), t=(t_1,\ldots,t_k),
\epsilon=(\epsilon_1,\ldots,\epsilon_k), M)$ is from now on the piece of
data with
which we represent a non-generic situation, and a polynomial system will
be associated to each tuple. Changing the order of the columns of $M$ amounts
to permuting the points $Q_j$.
Also, changing the sign of all the entries of a
column $j$ corresponds to taking the opposite of the point $Q_j$. While
it would not change the final complexity not to do so, it therefore makes
sense to consider only normalized tuples, in the sense that 
the columns of $M$
are sorted in lexicographical order, and the choice between a point $Q_j$ and its
opposite is done so that the sum of all elements in the corresponding
column is nonnegative. We remark that this is not enough to guarantee
that two normalized tuples do not describe similar situations. For
instance, if $\lambda=(1,\ldots,1)$ and two $t_i$ values are equal, then permuting the two
corresponding rows could lead to another normalized matrix that would
describe the same situation. This is not a problem for the general
algorithm: we might get the same $\ell$-torsion elements from two
different systems, but what is important to us is non-multiplicity (i.e.
radicality of the ideal) in each individual system.

\begin{definition}
    A {\em normalized non-genericity tuple} is a tuple
    $(w,\lambda,t,\epsilon, M)$, where
    $1\le w\le g$ is an integer, $\lambda=(\lambda_1,\ldots,\lambda_k)$ is a
    partition of $w$, $t$ and $\epsilon$ are vectors $t=(t_1,\ldots,t_k)$ 
    and $\epsilon=(\epsilon_1,\ldots,\epsilon_k)$ of the same length
    as $\lambda$
    with $1\le t_i\le g$ and $\epsilon_i\in\{0,1\}$, where $\epsilon_i$
    can be $1$ only if $t_i=1$ and $\lambda_i=1$,
    and finally
    $M$ is
    a matrix with $k$ rows and $s$ columns, where $0\le s\le g\,k$, and
    its entries are integers such that:
    \begin{itemize}
        \item For all $1\le i\le k$, the sum of the absolute values of the entries on the row
            $i$ is equal to $t_i$;
        \item The columns are sorted in lexicographical order;
        \item The sum of the rows of the matrix is a vector whose coordinates
	  are nonnegative.
    \end{itemize}
\end{definition}

From the discussion above, any $\ell$-torsion element is described by
(at least) one normalized non-genericity tuple.
In the following we will
give a polynomial system for each normalized non-genericity tuple, so
that all $\ell$-torsion elements described by it are modelled by this
system. Furthermore, the system will have the properties required to
apply Proposition~\ref{prop:mainpolsys}, so that the complexity result
will follow.

Before starting this, we discuss briefly a bound on the number of
normalized non-genericity tuples. Assuming everything is always of
maximal size, and not sorted, we have $g$ choices for $w$, then at most $g^g$
choices for $\lambda$ and $t$, at most $2^g$ choices for $\epsilon$, 
and finally at most $(g^{2g+1})^{g^2}$ choices for
$M$, which gives $g^{O(g^3)}$. As bad as it is, such a factor that depends
only on $g$ will not hinder the final complexity estimate in $O_g\left((\log
q)^{O(g)}\right)$, as explained in Section~\ref{sec:schoof}.

\paragraph{Non-generic division polynomials.}

The expression of $\lambda_i\ell \cdot (P_i-\infty)$ in Mumford representation
will be the same as in the generic case when its weight $t_i$ is equal to
$g$ and Lemma~\ref{lem:deg_delta} can be applied. But when $t_i$ is
strictly less than $g$, the weight-$g$ coordinate system is no longer
available; this is explicitly visible by the fact that the denominator
$e_g(x_i)$ of the coefficients of the $v$-polynomial vanishes.

Therefore we need to use a weight-$t$ coordinate system for describing a
non-generic divisor $\lambda_i\ell \cdot (P_i-\infty)$ in Mumford representation.
In this paragraph, in order to keep simple notation, we will work with
$\ell\cdot(P_i-\infty)$, keeping in mind that we do not impose any
condition on $\ell$, so that we can later replace $\ell$ by
$\lambda_i\ell$.

We consider, for $1\le t<g$, the set $V_{\ell,t}$ of points of the curve which are mapped to a
weight-$t$ divisor after multiplication by $\ell$:
$$ V_{\ell, t} = \left\{ (x,y)\in \C\ | \ \ell \cdot ((x,y)-\infty)\
\text{is of weight $t$} \right\}.$$
This is a (possibly empty) variety of dimension 0 that can be described
with the classical (generic) division polynomials of Cantor: we define
$$ \Delta_{\ell, t} = {\sf GCD}(\psi_\ell(x), \psi_{\ell-1}(x), \ldots,
\psi_{\ell-g+t+1}(x)),$$
so that $V_{\ell,t}$ is precisely the set of points $(x,y)$ for which
$\Delta_{\ell, t}(x) = 0$ and $\psi_{\ell-g+t}(x)\neq 0$, as stated by Cantor
in~\cite{Ca94} on page 141. The polynomial $\psi_\ell$ is essentially the
square root of the leading coefficient of $\delta_\ell$. It can be
computed efficiently and has degree in $O_g(\ell^2)$ by Theorem~8.17
of~\cite{Ca94}.
To avoid multiplicities, we define
$\tilde{\Delta}_{\ell, t}(x)$ the square-free polynomial whose roots are
exactly the roots of $\Delta_{\ell, t}(x)$ that are not roots of
$\psi_{\ell-g+t}(x)$. The degree of $\tilde{\Delta}_{\ell, t}(x)$ is again bounded by
$O_g(\ell^2)$. Furthermore since the points of $V_{\ell,t}$ come in
pairs of conjugate points sharing the same $x$-value, the degree of
$V_{\ell, t}$ is $2\deg\tilde{\Delta}_{\ell, t}(x)$.

\begin{definition}\label{def:nongendivpol}
    The non-generic division polynomials $\frak{u}_{\ell,t}$ and
    $\frak{v}_{\ell,t}$ are the polynomials in $X$ with coefficients in
    $\F_p[x,y]/(\tilde{\Delta}_{\ell, t}(x), y^2-f(x))$ such that
    $$ \ell \cdot ((x,y)-\infty) = \Big\langle \frak{u}_{\ell,t}(X),
    \frak{v}_{\ell,t}(X)\Big\rangle,$$
    in weight-$t$ Mumford representation: $\frak{u}_{\ell,t}(X)$ is monic
    of degree $t$, $\frak{v}_{\ell,t}(X)$ is of degree at most
    $t-1$ and they satisfy $\frak{u}_{\ell,t}\, |\, \frak{v}_{\ell,t}^2-f$.
\end{definition}

Just like for the classical division polynomials, the coefficients of
$\frak{u}_{\ell,t}(X)$ and of $\frac{1}{y}\, \frak{v}_{\ell,t}(X)$ are in
$\F_p[x]/\tilde{\Delta}_{\ell, t}(x)$ (they do not depend on $y$) and we
can choose representatives of them that are polynomials of degree less
than $\deg\tilde{\Delta}_{\ell, t}(x)$. Hence, the bounds given in
Lemma~\ref{lem:deg_delta} are also valid for the non-generic division
polynomials; and since there are no denominators in the coefficients of
$\frak{v}_{\ell,t}(X)$, the other part of Lemma~\ref{lem:deg_delta} also
holds trivially.

The non-generic division polynomials can be computed efficiently, once
the classical division polynomials are known: the polynomial
$\tilde{\Delta}_{\ell, t}(x)$ can be easily deduced, and then working in
the quotient algebra yields the result in a time $\softO_g(\ell^2)$,
which is negligible compared to the other parts of the algorithm.

\paragraph{Polynomial system derived from a normalized non-genericity
tuple.}

We now want to write a polynomial system whose solutions are the
$\ell$-torsion elements following a given normalized non-genericity
tuple $(w,\lambda,t,\epsilon,M)$.

First, we need variables for the coordinates of the $P_i$ such that the
$\ell$-torsion element is $D=\sum_{i=1}^k\lambda_i(P_i-\infty)$, with $P_i\ne
\pm P_j$ for all $i\ne j$.
As a consequence, we introduce $2k$ variables for the
coordinates $(x_i,y_i)$ of all the points $P_i$. Since these points are
on the curve, they verify $y_i^2=f(x_i)$, however if $P_i$ is a
ramification point this can be simplified into $y_i=0=f(x_i)$, which
avoids the multiplicities. We get a first set of equations
\begin{equation}\label{eq:eq1}\tag{Sys.1}
    \left\{
    \begin{array}{rl}
    y_i^2=f(x_i)\not=0,\quad &\text{for all $i$ in $[1,k]$ such that
    $\epsilon_i=0$}, \\
    y_i=f(x_i)=0,\quad &\text{for all $i$ in $[1,k]$ such that 
        $\epsilon_i=1$}.\\
    \end{array}\right.
\end{equation}

As we just discussed, we must model the fact that $P_i\ne \pm P_j$ for $i\ne
j$. This is done via the following set of inequalities:
\begin{equation}\label{eq:eq2}\tag{Sys.2}
  x_i\ne x_j,\quad \text{for all $i,j$ in $[1,k]$ such that $i\neq j$}. 
\end{equation}

The next step is to enforce the fact that the element
$\lambda_i\,\ell\cdot(P_i-\infty)$ is of weight $t_i$.
For the indices for which $t_i<g$, this is encoded by the equation
defining $V_{\lambda_i\ell,t_i}$:
\begin{equation}\label{eq:eq3}\tag{Sys.3}
    \left\{
    \begin{array}{l}
    \tilde{\Delta}_{\lambda_i\ell, t_i}(x_i) = 0,\\
    d_{t_i}(x_i) \ne 0,\\
    \end{array}
    \right.
        \quad \text{for all $i$
    in $[1,k]$ such that $t_i<g$},
\end{equation}
while for the indices for which $t_i=g$, this is encoded by the
non-vanishing of the leading coefficient of the Cantor polynomial in
degree $\lambda_i\ell$:
\begin{equation}\label{eq:eq4}\tag{Sys.4}
    d_{g}(x_i) \ne 0 ,\quad \text{for all $i$ in $[1,k]$ such that
    $t_i=g$}.
\end{equation}

We now need to model the fact that the $\lambda_i\,\ell\cdot(P_i-\infty)$
satisfy the conditions given by the matrix $M$. We write
$\lambda_i\,\ell\cdot(P_i-\infty)=\langle u_i(X), v_i(X)\rangle$ in
Mumford representation, where $u_i(X)$ and $v_i(X)$ are Cantor's classical
division polynomials in degree $\lambda_i\ell$ if $t_i=g$ or the
non-generic division polynomials
$\frak{u}_{\lambda_i\ell,t_i}$ and $\frak{v}_{\lambda_i\ell,t_i}$, 
if $t_i<g$. In both cases, these are polynomials in $X$ whose coefficients are
polynomials in $x_i$ and $y_i$. Recall that the entries of $M$, denoted by
$(m_{ij})_{i\in[1,k], j\in[1,s]}$, are such that $m_{ij}$ is the order of $Q_j$
in $\lambda_i\,\ell\cdot(P_i-\infty)$ if it is positive, or the opposite of the
order of $Q_j'$ if it is negative. To this effect, we introduce $s$ new
variables $\xi_j$ for the abscissae of the $Q_j$, and the following
equations enforce the multiplicities:
\begin{align}
    u_i^{(n)}(\xi_j)=0,\quad
        & \text{for all $i,j$ in $[1,k]\times[1,s]$ and for all $n \le
        \vert m_{ij}\vert-1$}\label{eq:eq5}\tag{Sys.5} \\
    u_i^{(\vert m_{ij}\vert)}(\xi_j) \ne 0,\quad
        & \text{for all $i,j$ in $[1,k]\times[1,s]$}\label{eq:eq6}\tag{Sys.6} \\
    v_i(\xi_j)-v_{i'}(\xi_j)=0,\quad
        & \text{for all $i,i',j$ such that $
        m_{ij}m_{i'j}>0$}\label{eq:eq7}\tag{Sys.7} \\
    v_i(\xi_j)+v_{i'}(\xi_j)=0,\quad
        & \text{for all $i,i',j$ such that $
        m_{ij}m_{i'j}<0$}\label{eq:eq8}\tag{Sys.8} \\
    \xi_j \not = \xi_{j'},\quad &
        \text{for all $j\not=j'$}.\label{eq:eq9}\tag{Sys.9}
\end{align} 
In Equations~\ref{eq:eq5} and~\ref{eq:eq6}, the notation $u_i^{(n)}$ is
for the $n$-th derivative of $u_i$. This simple way of describing
multiple roots is valid because the characteristic is large enough.

The next step of the construction is to consider a semi-reduced version
of the divisor $\ell\cdot D = \sum_{i=1}^k\lambda_i\,\ell\cdot(P_i-\infty)$. This
semi-reduction process can be described directly on the matrix $M$: if
two entries in a same column have opposite signs, a semi-reduction can
occur (corresponding to subtracting the principal divisor of the function
$(x-\xi_j)$), thus reducing the difference between these entries. This
semi-reduction can continue until one of these two entries reaches zero.
This whole process can be repeated as long as there are still columns
containing entries with opposite signs. This is formalized in
Algorithm~\ref{algo:semireduction}, which takes as input a matrix $M$ and
returns a matrix $\widetilde{M}$ with the same dimensions such that if $M$ describes all the
multiplicities in a divisor, then $\widetilde{M}$ describes all the
multiplicities of a semi-reduced divisor equivalent to the input divisor.
More precisely, the matrix $\widetilde M$ satisfies the following properties:
(1) In each column, all elements are nonnegative; (2) The sum of
the rows of $M$ equals the sum of the rows of $\widetilde M$; (3) For all $i,j$
such that $M_{i,j}$ is nonnegative,
$\widetilde M_{ij}\leq M_{ij}$.

\begin{algorithm}[ht]
 \KwData{$M$ the $k\times s$ matrix of shared points of the system}
 \KwResult{$\widetilde{M}$, the matrix after semi-reduction}
 $\widetilde M\gets$ $k\times s$ zero matrix \\
 \For{$j$ from $1$ to $s$}{
   $\mu_j\leftarrow \sum_{i=1}^k M_{ij} $ \\
   \For{$i$ from $1$ to $k$}{
     \eIf{$M_{ij} >0$}{
       $\widetilde{M}_{ij} \leftarrow \min(
       M_{ij},\mu_j)$ \\
       $\mu_j\leftarrow \mu_j-\widetilde M_{ij}$}
       {$\widetilde{M}_{ij} \leftarrow 0$}
   }
 }
 \Return{$\widetilde{M}$}
 \caption{Reducing the matrix of shared points
   \label{algo:semireduction}}
\end{algorithm}

The function $\varphi$ that we will use to model the principality of the
divisor $\ell\cdot D$ will have two parts: a product of ``vertical
lines'' corresponding to semi-reductions, and a part of the form
$P(X)+YQ(X)$, where $P$ and $Q$ are coprime.  Modelling the existence of
this second part requires to introduce new entities $\widetilde{u}_i$ that
are the $u_i$ polynomials from which we remove the linear factors coming
from semi-reduction as described by $\widetilde{M}$. Formally, we have
the following equations, defining $\widetilde{u}_i$:
\begin{equation}\label{eq:eq10}\tag{Sys.10}
    u_i(X) = \widetilde{u}_i(X) \prod_{j=1}^s (X-\xi_j)^{\vert m_{ij}
    \vert - \widetilde{m}_{ij}},\quad
    \text{for all $i\in[1,k]$}.
\end{equation}
Indeed, by definition of the matrix $M$, the factor $(X-\xi_j)^{\vert
m_{ij}\vert}$ divides exactly $u_i(X)$, and the factor $(X-\xi_j)^{
\widetilde{m}_{ij}}$ divides exactly $\widetilde{u}_i(X)$. In order
to express these conditions efficiently in the polynomial system, we
introduce new variables for the coefficients of the $\widetilde{u}_i$
polynomials.

Since we are now dealing with a semi-reduced divisor, we can consider its
Mumford representation, {\em i.e.} two polynomials $U$ and $V$ with the
following properties:
\begin{align}
    & U = \prod_{i=1}^k \widetilde{u}_i, \quad U | V^2-f,
    \label{eq:eq11}\tag{Sys.11} \\
    & V \equiv v_i \bmod \widetilde{u}_i, \quad
        \text{for all $i\in[1,k]$}. \label{eq:eq12}\tag{Sys.12}
\end{align}
The expression of $U$ is simple enough, so we do not have to introduce
new variables for its coefficients. However, this will be necessary for
the coefficients of the $V$ polynomial. Finally, in order to impose that
the semi-reduced part of $\varphi$ has exactly the zeros described by
this divisor, we have the equation
\begin{equation}\label{eq:eq13}\tag{Sys.13}
    P + QV \equiv 0 \bmod U,
\end{equation}
which is expressed with new variables for the coefficients of $P$ and
$Q$.

In Table~\ref{tab:variables}, we summarize all the variables used in the
polynomial system and count them. A key quantity for this count is the
degree of $U$ which is the sum of the degrees of the $\widetilde{u}_i$'s.
It can be computed directly from the tuple $(w,\lambda,t,\epsilon,M)$. Then, to ensure
existence and unicity of the $V$ polynomial to represent the semi-reduced
divisor, we have to impose that $\deg V < \deg U$, so that we have
exactly $\deg U$ variables for the coefficients of $V$. For the
polynomials $P$ and $Q$, we need the degree of $P^2-Q^2f$ to be exactly
$\deg U$. After a normalization like in Section~\ref{sec:generic}
depending on the parity of $\deg U$, we get $\deg U-g$ variables for
their coefficients.

\begin{table}[ht]
    \centerline{%
    \begin{tabular}{lll}
        \hline
        Variables & Number of variables & Bound\\
        \hline
        Coordinates $(x_i, y_i)$ of $P_i$  & $2k$ & $2g$ \\
        Abscissae $\xi_j$ of shared points & $s$, column-size of the matrix $M$ & $g^2$\\ 
        Coefficients of the $\widetilde{u}_i$ polynomials & $\deg U = \sum_i (t_i
        - \sum_j (\vert m_{ij} \vert - \widetilde{m}_{ij}) )$ & $g^2$\\
        Coefficients of the $V$ polynomial & $\deg U$ & $g^2$ \\
        Coefficients of the $P$ and $Q$ polynomials & $\deg U - g$ &
        $g^2-g$ \\
        \hline
        Total   & $s+2k+3\deg U-g$ & $4g^2+g$ \\
        \hline
    \end{tabular}}
    \caption{Summary of the variables in the polynomial system
    corresponding to a normalized non-genericity tuple
    $(w,\lambda,t,\epsilon,M)$.}
    \label{tab:variables}
\end{table}

In order to apply Proposition~\ref{prop:mainpolsys}, we need to evaluate
the degrees of all the equations and inequalities that we have listed,
with respect to two groups of variables: The first group contains just
the variables $x_i$ and $y_i$, and we will denote $\deg_1(f)$ the
degree of a polynomial $f$ with respect to those variables (said otherwise,
$\deg_1(f)$ is the degree of $f$ if we consider only the symbols $x_i, y_i$ as
variables, and all the other indeterminates are considered as parameters). The
second group of variables contains all the
other indeterminates and the degree with respect to this group is denoted by
$\deg_2$.

The crucial point is to ensure that each polynomial equation has a $\deg_1$
bounded by $O_g(\ell^3)$, while $\deg_2$ is bounded by $O_g(1)$.  For the
inequalities, we require the same degree conditions: Indeed, an inequality
$f\ne 0$ can be modeled by the equality $T\cdot f - 1 = 0$, where $T$ is a
fresh variable that belongs to our second group of variables. This trick
requires only one more variable for each inequality and the degree of the
equation $T\cdot f -1 = 0$ is only one more than the degree of the inequality.
Since the number of inequalities is bounded by $O_g(1)$, the number of extra
variables required in the second group will not impact the asymptotic
complexity (the second group already contains $O_g(1)$ variables). We
remark that the input of the geometric resolution algorithm over fields of characteristic
$0$ in \cite{GiuLecSal01} allows inequalities. However, we use the
aforementioned trick to model inequalities by equalities since the solving method that we use is
the variant for positive characteristic whose complexity analysis is given in
\cite[Thm.~4.8]{cafure2006fast}.

The number of equations and inequalities and their degrees with respect to the
two groups of variables can be easily checked and are summarized in Table~\ref{tab:degrees}. 

\begin{table}[ht]
    \centerline{%
    \begin{tabular}{llll}
        \hline
        Equations reference & Number of equations (and bound) & $\deg_1$ & $\deg_2$ \\
        \hline
        Eq. and Ineq. \ref{eq:eq1} & $2k\le 2g$  & $2g+1$ & $0$ \\
	InEq. \ref{eq:eq2} & $k(k-1)/2\le g(g-1)/2$ & $1$ & $0$ \\
        Eq. and Ineq. \ref{eq:eq3}   & $\le 2g$ & $O_g(\ell^3)$ & $0$ \\
        InEq. \ref{eq:eq4} & $\le g$ & $O_g(\ell^3)$ & $0$ \\
        Eq. \ref{eq:eq5}   & $\sum_{i=1}^k \sum_{j=1}^s |m_{ij}|\le g^4$ &
                    $O_g(\ell^3)$ & $\le g$ \\
        InEq. \ref{eq:eq6} & $ks\le g^3$ & $O_g(\ell^3)$ & $\le g$ \\
        Eq. \ref{eq:eq7} and \ref{eq:eq8} & $\le k^2s\le g^4$ & $O_g(\ell^3)$ & $\le g$\\
        InEq. \ref{eq:eq9} & $\le s^2\le g^4$ & $0$ & $1$ \\
        Eq. \ref{eq:eq10}  & $\sum_{i=1}^k t_i\le g^2$ & $O_g(\ell^3)$ & $\le g$\\
        Eq. \ref{eq:eq11}  & $\deg U \le g^2$ & $0$ & $O(g^3)$ \\
        Eq. \ref{eq:eq12} & $\sum_{i=1}^k \deg \widetilde{u}_i \le g^2$ &
                    $O_g(\ell^3)$ & $O(g^2)$\\
        Eq. \ref{eq:eq13} & $\deg U\le g^2$ & $0$ & $O(g^3)$ \\
        \hline
    \end{tabular}}
    \caption{Summary of the degrees of the equations in the polynomial
    system corresponding to a normalized non-genericity tuple
    $(w,\lambda,t,\epsilon,M)$.}
    \label{tab:degrees}
\end{table}

Finally, since we have been very careful in describing elements that are
$\ell$-torsion points on $J$, without room for parasite solutions or
multiplicities, we can again appeal to the finite and étale property of
multiplication by $\ell$ in $J$ to deduce that the system is
$0$-dimensional and radical. Therefore, by
Proposition~\ref{prop:mainpolsys}, each system can be solved in the
claimed complexity bound.
To conclude the proof of Proposition~\ref{prop:Jl}, and hence of our main result, we
need a few more observations.

First, notice that the solutions of our polynomial systems can be grouped by
weight of the $\ell$-torsion divisor: once a geometric resolution of two
$0$-dimensional sets $V_1$ and $V_2$ are known, a geometric resolution of
$V_1\cap V_2$ can be computed very efficiently. The strategy to do so is to
change the primitive element of the geometric resolutions for a random element,
so that both resolution share the same primitive element. This can done within
complexity linear in the number of variables and polynomial in $\deg(V_1\cup
V_2)$ using
Algorithm~6 in \cite{GiuLecSal01}. Then, computing the LCM of the univariate
polynomials of the geometric resolutions and interpolating the parametrization provides a
geometric resolution of $V_1 \cup V_2$.
Using this procedure for regrouping the solutions of all the systems derived
from the non-degeneracy tuples with the same weight $w$ provides geometric
solutions of $J_w[\ell]$ within the claimed complexity.

Finally, we need to
transform the Monte Carlo algorithm from Proposition~\ref{prop:mainpolsys} in a
Las Vegas algorithm. This can be easily achieved since the probability that the
Monte Carlo algorithm succeeds is bounded below by a quantity which does not
depend on the input size, and the output can be verified since we know that
the sum of the degrees of the varieties $J_w[\ell]$ for $w\in[1,g]$ must equal
$\ell^{2g}-1$. Consequently, once all polynomial systems corresponding to
non-generic situations have been solved, it is easy to count the number of
$\ell$-torsion elements found and to check that none of them is missing by
comparing their number with the theoretical value $\ell^{2 g}-1$. The Las Vegas algorithm consists in repeating the
Monte Carlo algorithm until the result is verified and is correct (i.e. all
elements found are $\ell$-torsion elements and none of them is missing). The expected
complexity of the Las Vegas variant equals the complexity of the Monte Carlo variant up
to multiplication by a constant. This concludes the proof of
Proposition~\ref{prop:Jl}.

\section{Proof of Lemma \ref{lem:deg_delta}}\label{sec:proof}   

We restate Lemma \ref{lem:deg_delta}: for a non-necessarily prime integer
$\ell>g$, the polynomial $\delta_{\ell}(X)$
of degree $g$ in $X$ has coefficients in $\F_q[x]$ whose degrees in $x$
are bounded by $g\ell^3/3+\Og(\ell^2)$; the polynomial
$\varepsilon_{\ell}(X)/y$ has coefficients in $\F_q(x)$ such that the
degrees of the numerators and the denominators have degrees bounded by
$2g\ell^3/3+\Og(\ell^2)$.  Furthermore, the roots of the denominators are
roots of the leading coefficient of $\delta_\ell(X)$.

This section strongly relies on \cite{Ca94}. From here, all the
notations, references to propositions, expressions or definitions are
taken from this paper, except for the notation of the base field that we
choose to be $\F_q$, and the torsion level is denoted $\ell$ instead of
$r$.

\begin{proof}
    Technicalities arise from the normalizations required to manipulate
    entities that are polynomials in $x$ (and not rational fractions),
    without odd power of $y$ involved. In Cantor's article, this
    normalization often depends on the parity of $\ell-g$. We will
    concentrate on the case where $g$ is even; for the other case some
    formulae must be adapted, multiplying or dividing by $2y$ at various
    places.

    We recall that $\nu_\ell=(\ell^2-\ell-g^2+g)/2$ as defined in (8.5),
    so that $\nu_\ell=\nu_{\ell-1} + \ell-1$.
    By combining (8.7) and (2.3), we obtain
    $$\delta_{\ell}(z)= \frac{(2y)^{2\nu_\ell}}{(4y^2z)^\ell}
    \left( A_\ell(4y^2z)^2 - B_\ell(4y^2z)^2E(4y^2z) \right),$$
    where $A_\ell$ and $B_\ell$ are unnormalized versions of
    $\alpha_\ell$ and $\beta_\ell$ defined in (8.7) and $E(z)$ is
    defined by $E(z)=f(x-z)$. For our purpose, it is easier to deal with
    non-truncated versions of $\alpha_\ell$ and $\beta_\ell$. Let us then
    introduce the following quantities, inspired from (8.7):
    $$ \bar{\alpha}_\ell(z) = 2(2y)^{\nu_{\ell-1}-1}A_\ell(4y^2z), \quad
    \text{and}\quad \bar{\beta}_\ell(z) = (2y)^{\nu_{\ell-1}}B_\ell(4y^2z),$$
    so that $\delta_\ell$ can be rewritten as
    $$ \delta_\ell(z) = 
    \frac{1}{4z^\ell}\left(
    \bar{\alpha}_\ell(z)^2 - \frac{1}{y^2}\bar{\beta}_\ell(z)^2E(4y^2z)
    \right).$$
    By Theorem (8.15), the coefficients of $\alpha_\ell(z)$ and
    $\beta_\ell(z)$ are polynomials in $\F_q[x]$, and the proof is also
    valid for the non-truncated versions $\bar{\alpha}_\ell(z)$ and
    $\bar{\beta}_\ell(z)$. Note that here we use the fact that $g$ is
    even, so that the potential adjusting factor $(2y)^g$ is an even
    power of $y$ that can be rewritten in terms of $f(x)$.
    The polynomial $E(4y^2z)$ has coefficients which are polynomials in
    $x$ of degree bounded by $(2g+1)^2$. Therefore, in order to obtain
    a degree bound for the
    coefficients of $\delta_\ell(z)$, it is sufficient to bound the
    coefficients of $\bar{\alpha}_\ell(z)$ and $\bar{\beta}_\ell(z)$.
    For convenience, for any polynomial $P$ in $z$, with coefficients in
    $\F_q[x]$, we write $\MDC(P)$ the maximum degree in $x$ of its
    coefficients.

    We are interested in a bound at fixed genus $g$ and when $\ell$ grows
    to infinity and we use the $\Og$ notation. For $k$ in $[1,\ell]$, we
    will use an induction to bound $\MDC(\bar{\alpha}_k(z))$ and
    $\MDC(\bar{\beta}_k(z))$. For $k\le g+1$, none of these quantities
    depends on
    $\ell$, so that all the degrees can be bounded by an expression in
    $g$ only, {\em i.e.} in $\Og(1)$. For $k\ge g+1$, we start from the
    equation (3.14) where we substitute $k$ for $r$, we evaluate it at
    $4y^2z$, and we multiply by $(2y)^{2\nu_k-1}$, so that we
    obtain:
    $$ (2y)^{\nu_k}f_{k-1}\bar\alpha_{k+1}(z) =
       (2y)^{\nu_k+k-1}f_{k}\bar\alpha_{k}(z) -
       (2y)^{\nu_k+2k-1}f_{k+1}z\bar\alpha_{k-1}(z),$$
    where all the polynomials have coefficients in $\F_q[x]$. The
    expression for $\bar{\beta}_k$ is exactly the same, but we have to
    multiply the expression in (3.14) by $(2y)^{2\nu_k}$ in that case.
    By (8.7), (8.16)
    and Theorems (8.15) and (8.17), for any $k$, the quantity
    $(2y)^{\nu_k}f_k$ is a polynomial in $x$ of degree $g(k^2-g^2)/2$. Therefore the
    right-hand-side of the recurrence relation have coefficients with
    degrees bounded by an expression of the form
    $\max\big(\MDC(\bar\alpha_k(z)),\MDC(\bar\alpha_{k-1}(z))\big) +
    gk^2/2$ up to a term linear in $k$ and cubic in $g$. And we finally get:
    $$ \MDC(\bar\alpha_{k+1}(z)) \le \max\big(\MDC(\bar\alpha_k(z)),
    \MDC(\bar\alpha_{k-1}(z))\big) + gk^2/2 + \Err_g(k),$$
    where $\Err_g(k)$ is a polynomial linear in $k$ and cubic in $g$.
    Again, this inequality is also valid for $\bar\beta_k$.
    By induction, we then get the following bounds:
    $$ \MDC(\bar\alpha_{\ell}(z)) \le \frac{g\ell^3}{6} + \Og(\ell^2),\quad
    \text{and}\quad
    \MDC(\bar\beta_{\ell}(z)) \le \frac{g\ell^3}{6} + \Og(\ell^2).$$
    We can then propagate these bounds in the expression of $\delta_\ell$
    and we get $\MDC(\delta_\ell(z)) \le \max (2\MDC(\bar\alpha_{\ell}(z)),
    2\MDC(\bar\beta_{\ell}(z)) + \MDC(E(4y^2z))$, so that we get the
    claimed result concerning $\delta_\ell$.
    \bigskip

    The fact that $\varepsilon_\ell(z)/y$ has coefficients in $\F_q(x)$
    follows directly from Equation (8.13) that we recall here:
    $$
    \varepsilon_{\ell}(z) = y \frac
    {z \left(\psi^2_{\ell-1}\delta_{\ell+1}(z)
       - \psi^2_{\ell+1}\delta_{\ell-1}(z)\right)}
    {\psi_{\ell-1}\psi^2_{\ell}\psi_{\ell+1}}
    \bmod \delta_{\ell}(z).$$
    By (8.11), the leading coefficient of $\delta_\ell(z)$ is
    $-(4y^2)^g\psi_\ell^2$, so that the property on the denominator of
    $\varepsilon_\ell$ can not be easily deduced from this equation, due
    to the presence of $\psi_{\ell-1}$ and $\psi_{\ell+1}$ before the
    reduction modulo $\delta_\ell(z)$ occurs. We will prove it below,
    with a direct geometric argument, but we first give bounds on the
    degrees of the coefficients of the numerator and the denominator.

    The polynomial $\delta_\ell(z)$ is of degree $g$ in $z$, so that at
    most two steps of reduction are required to reduce the degree of
    $\varepsilon_\ell$ to strictly less than $g$. In fact,
    it can be checked that the leading coefficient of
    $\psi^2_{\ell-1}\delta_{\ell+1}(z)-\psi^2_{\ell+1}\delta_{\ell-1}(z)$
    is zero, so that there is at most only one reduction step.
    This reduction accounts for an increase of the
    coefficients' degrees in $x$ by at most $\MDC(\delta_\ell)$ in the
    numerator and an increase of the degree of the leading coefficient of
    $\delta_\ell$ in the denominator. Since $\deg \psi_\ell = g\ell^2/2 +
    O_g(\ell)$, the degrees of the
    coefficients of the numerator of $\varepsilon_\ell(z)$ are bounded by
    $\frac{2}{3}g\ell^3 + O_g(\ell^2)$, and the degree of the denominator
    is bounded by $3g\ell^2+O_g(\ell)$.

    It remains to prove the claim on the roots of the denominator of the
    coefficients of $\varepsilon_\ell(z)/y$. For this, we consider the
    map from the affine part of the curve $\C_{\aff}$ to $J$ seen as a
    projective Abelian variety, that sends a point $(x,y)$ to
    $[\ell]((x,y)-\infty)$. One of the main points of Cantor's article is
    that if $\psi_\ell(x)\ne 0$, then the image by this map is in
    $J\setminus \Theta$, where $\Theta\subset J$ is the subvariety of elements
    of weight less than $g$. On this open subset, Mumford coordinates with
    a monic $u$ of degree $g$ and $v$ of degree at most $g-1$ give a
    local set of coordinates that we use to describe the map.
    The $i$-th coefficient of $v$ is $y$ times
    a rational fraction $c_i$ in $x$ that gives a finite value at any $x$
    for which $\psi_\ell(x)\ne 0$.  Therefore, any root of the
    denominator of $c_i$ is a root of $\psi_\ell$. By Theorem (8.35),
    the Mumford $v$-polynomial that we are considering is
    $\varepsilon_\ell$ up to a renormalization that will only introduce
    additional powers of $4y^2$ in the denominator. Therefore, any root
    of the denominator of the coefficients of $\varepsilon_\ell$ is a
    root of $\psi_\ell$ or of $4y^2$, and both divide the leading
    coefficient of $\delta_\ell$ by (8.11).
\end{proof}

\paragraph*{Remark.} 

The bounds that we obtain are not tight: from~\cite{Ca94}, we know that
the leading and constant coefficients are in $O_g(\ell^2)$ instead of
$O_g(\ell^3)$. We ran experiments that allow us to conjecture precise
degrees for the other coefficients.
In these experiments, instead of developing $\delta_{\ell}\left(\frac{x-X}{4y^2} \right)$
 and $\varepsilon_{\ell}\left(\frac{x-X}{4y^2} \right)$ to compute the $d_i$'s 
and $e_i$'s, we computed $\ell\cdot((x,y)-\infty)$ over the function field of the curve. This
does not exactly yield the $d_i$'s and $e_i$'s because we actually get $d_i/d_g$
and $e_i/e_g$, thus possibly missing a common factor in all the $d_i$'s and 
$e_i$'s.
We denote $\tilde{d}_i$ and $\tilde{e}_i$ the numerators and 
denominators of the aforementioned fractions, and we compute their degrees for
each pair $(g,\ell)$ with $g\le 8$ and $g<\ell \le g+20$ (which includes non 
prime values of $\ell$). We found that the degrees of the $ \tilde{d}_i$ are 
consecutive from $\deg( \tilde{d}_g)$ up to $\deg( \tilde{d}_0)=\deg( \tilde{d}_g)+g$, 
with the following values for $\deg( \tilde{d}_0)$.

$$ \left\lbrace  \begin{aligned}
   & g\ell^2-g^3+g & \mbox{if}\: g-\ell\: \mbox{is even} \\
   & g\ell^2-g^3+2g^2-1 & \mbox{if}\: g-\ell\: \mbox{is odd}
\end{aligned} \right. $$

Concerning the $ \tilde{e}_i$, the degrees are consecutive from $\deg( \tilde{e}_{g-1})$ up to
$\deg( \tilde{e}_0)=\deg( \tilde{e}_g)$, the latter being equal to 
 $$ \left\lbrace  \begin{aligned}
   & 3(g\ell^2-g^3)/2+2g^2-g-1 & \mbox{if}\: g-\ell\: \mbox{is even} \\
   & 3(g\ell^2-g^3)/2+3g^2-g/2-1 & \mbox{if}\: g-\ell\: \mbox{is odd}
\end{aligned} \right. $$

Cantor~\cite{Ca94} gave simple expressions for the leading term and constant term of
$\delta_{\ell}$ (respectively $-(4y^2)^g\psi_{\ell}^2$ and
$(-1)^{g+1}\psi_{\ell-1}\psi_{\ell+1}$), from which we can deduce the
degrees of $d_0$ and $d_g$ by evaluating $\delta_\ell$ at $(x-X)/4y^2$.
Assuming that there is no common factor to all the $d_i$'s when $g-\ell$
is even, while the GCD of all the $d_i$'s is $f^{g-1}$ when $g-\ell$ is
odd, these theoretical degrees are consistent with our experiments.

\bibliographystyle{abbrv}
\bibliography{./biblio}
\end{document}